\documentclass[10pt]{article}
\usepackage[a4paper, margin=1in]{geometry}
\usepackage[utf8]{inputenc}
\usepackage[T1]{fontenc}
\usepackage{lmodern}
\usepackage{textcomp}
\usepackage{mathrsfs}
\usepackage{authblk}

\usepackage{amssymb, latexsym, amsmath, amsthm}
\usepackage{mathtools}

\usepackage{verbatim}
\usepackage{setspace}
\usepackage[super]{nth}
\usepackage{titlesec}
\usepackage{enumitem}
\usepackage[useregional]{datetime2}

\usepackage{tikz}
\usepackage{tikz-cd}

\usepackage{hyperref}

\newtheorem{claim}{Claim}

\theoremstyle{plain}
\newtheorem{theorem}{Theorem}[section] 

\newtheorem{lemma}[theorem]{Lemma}
\newtheorem{proposition}[theorem]{Proposition}
\newtheorem{corollary}[theorem]{Corollary}
\newtheorem{conjecture}[theorem]{Conjecture}
\newtheorem{definition}[theorem]{Definition}
\newtheorem{remark}[theorem]{Remark}

\theoremstyle{definition}

\DeclareMathOperator{\Hom}{Hom}
\DeclareMathOperator{\im}{im}

\DeclareMathOperator{\id}{id}
\DeclareMathOperator{\Sym}{Sym}
\DeclareMathOperator{\GL}{GL}

\DeclareMathOperator{\Gal}{Gal}
\DeclareMathOperator{\Kaz}{Kaz}
\DeclareMathOperator{\Vol}{Vol}
\DeclareMathOperator{\Top}{top}
\DeclareMathOperator{\Lie}{Lie}
\DeclareMathOperator{\rank}{rank}
\DeclareMathOperator{\Del}{Del}
\DeclareMathOperator{\codim}{codim}
\DeclareMathOperator{\Irr}{Irr}
\DeclareMathOperator{\unr}{unr}
\DeclareMathOperator{\Fr}{Fr}
\DeclareMathOperator{\tor}{tor}
\DeclareMathOperator{\der}{der}

\DeclareMathOperator{\Red}{red}
\DeclareMathOperator{\Aut}{Aut}

\DeclareMathOperator{\kur}{F_{\text{unr}}}

\DeclareMathOperator{\Bred}{\mathcal{B}_{\text{red}}}

\DeclareMathOperator{\Spec}{Spec}
\DeclareMathOperator{\val}{val}
\DeclareMathOperator{\Pic}{Pic}
\DeclareMathOperator{\WD}{WD}

\DeclareMathOperator{\LLC}{LLC}
\DeclareMathOperator{\qs}{qs}
\DeclareMathOperator{\depth}{depth}
\DeclareMathOperator{\SL}{SL}
\DeclareMathOperator{\Ad}{Ad}
\DeclareMathOperator{\Sp}{Sp}
\DeclareMathOperator{\SO}{SO}
\DeclareMathOperator{\GSp}{GSp}
\DeclareMathOperator{\Char}{char}
\DeclareMathOperator{\pin}{pin}
\DeclareMathOperator{\Tr}{\mathrm{Tr}}
\DeclareMathOperator{\Pin}{\mathfrak{pin}}

\newcommand{\SchP}{\mathcal{P}}
\newcommand{\SchQ}{\mathcal{Q}}

\numberwithin{equation}{section}

\renewcommand{\thesubsection}{\thesection.\Alph{subsection}}

\titleformat{\subsection}[runin]
  {\normalfont\bfseries}    
  {\thesubsection.}         
  {0.5em}                   
  {\itshape}                
  [:]

\title{The formal degree conjecture for groups over local function fields}

\author{Anantha Krishna B}
\affil{Department of Mathematics, Indian Institute of Science, Bengaluru, India\\
\texttt{ananthakb@iisc.ac.in}}

\date{}

\begin{document}

\maketitle

\begingroup
\renewcommand\thefootnote{}
\endgroup

\begin{abstract}
    In this article, we will prove that the formal degree conjecture is compatible with the Deligne-Kazhdan correspondence for quasi-split groups, assuming that the local Langlands correspondence is compatible with the Deligne-Kazhdan correspondence. Consequently, we establish the formal degree conjecture for $\GL_n$ over local
    function fields of characteristic $p > 0$, and for $\Sp_{2n}$, split $\SO_{2n}$,
    $\SO_{2n+1}$, and $\GSp_4$ over local function fields of characteristic $p > 2$.
\end{abstract}

\section{Introduction}

Let $F$ be a non-Archimedean local field and $G$ be a connected, reductive group over $F$. Let $Z$ be the split component of the center of $G$. For a complex discrete series representation $(\pi,V)$ of $G(F)$, there exists a positive real number $d(\pi,\mu)$ which depends on the chosen Haar measure $\mu$ on $Z(F) \backslash G(F)$ such that 

$$\int_{Z(F)\backslash G(F)} \langle \pi(g)v,v^{\vee} \rangle {\langle \pi(g^{-1})w,w^{\vee}\rangle} \, d\mu = d(\pi, \mu)^{-1} \langle w,v^{\vee} \rangle {\langle v,w^{\vee} \rangle},$$ 

for all $v,w \in V$ and $v^{\vee}, w^{\vee} \in V^{\vee}$, where $V^{\vee}$ is the contragredient of $V$ and $\langle \cdot, \cdot \rangle$ is the natural pairing between $V$ and $V^{\vee}$.
Assuming the local Langlands conjecture (LLC), the formal degree conjecture was formulated in \cite{HiragaIchinoKaoru08} and \cite{HiragaIchinoIkeda08_B}. The conjecture gives an explicit formula for the formal degree of a discrete series representation in terms of Langlands parameters and adjoint $\gamma$-factors. To formulate the conjecture, we normalize the Haar measure on $Z(F) \backslash G(F)$ as in \cite{HiragaIchinoIkeda08_B} by choosing a non-trivial additive character $\psi$ of $F.$ We denote this measure by $\mu_{Z(F) \backslash G(F), \psi}$ or simply $\mu_{Z \backslash G}$ whenever the base field $F$ and the character $\psi$ are clear from the context.
The formal degree conjecture has been established in many cases for groups over characteristic $0$ fields, including $\GL_n$ in \cite{HiragaIchinoKaoru08}, $\SO_{2n+1}$ in \cite{IchinoAtsushietc17}, $\GSp_4$ in \cite{GanIchino14}, and $\Sp_{2n}$ and $\SO_{2n}$ in \cite{beuzartplessis2025}.

To resolve the formal degree conjecture in the positive characteristic case, we use the Deligne-Kazhdan correspondence. Let us briefly recall this:
\begin{enumerate}[label = \roman*)]
    \item For any local field $F$ of positive characteristic and an integer $m \ge 1,$ there exists a local field $F'$ of characteristic $0$ such that $F$ is $m$-close to $F'$ (see Definition \ref{def:close_field}).
  \item Deligne \cite{Deligne84} proved that if $F$ and $F'$ are $m$-close, then there exists an isomorphism $\Gal(F_s/F) /I_F^m \xrightarrow{\simeq} \Gal(F'_s/F')/I_{F'}^m,$ where $I_F^m$ is the $m$-th upper ramification subgroup of the inertia group $I_F.$ This induces a bijection between the sets of isomorphism classes:
\begin{gather*}
    \left\{
        \begin{tabular}{c}
            Continuous, complex, finite-dimensional \\
            representations of $\Gal(F_s/F)$ trivial on $I_F^m$
        \end{tabular}
    \right\}
    \\ \updownarrow \\
    \left\{
        \begin{tabular}{c}
            Continuous, complex, finite-dimensional \\
            representations of $\Gal(F'_s/F')$ trivial on $I_{F'}^m$
        \end{tabular}
    \right\}.
\end{gather*}
Moreover, all of the above holds when $\Gal(F_s/F)$ is replaced by the Weil group $W_F$ or the Weil-Deligne group $\WD_F.$

\item \label{item:Kazdhan} Let $G$ be a split, connected, reductive group over $\mathbb{Z}$. Kazhdan \cite{Kazdhan86} proved that for any $m \ge 1,$ there exists an integer $l \ge m$ such that if $F$ and $F'$ are $l$-close, there is an isomorphism of Hecke algebras $\mathcal H(G(F),K_m) \xrightarrow{\sim} \mathcal H(G(F'), K_m'),$ where $K_m$ denotes the $m$-th principal congruence subgroup of $G(\mathfrak{o}_F).$

\item Ganapathy \cite{Ganapathy22} generalized Item \ref{item:Kazdhan} to general connected reductive groups. Specifically, given a group $G$ over $F$ which splits over an at most $m$-ramified extension, and if $F$ and $F'$ are $m$-close, then one can construct a connected reductive group $G'$ over $F'$.(see Sections 3 and 5 of \cite{Ganapathy19}). Ganapathy proved in \cite{Ganapathy22} that for any $m \ge 1,$ there exists an integer $l \ge m$ such that if $F$ and $F'$ are $l$-close, there is a Hecke algebra isomorphism $\Kaz_m : \mathcal H(G(F),P_m) \xrightarrow{\sim} \mathcal H(G'(F'), P_m'),$ where $P_m$ is the $m$-th congruence subgroup of a parahoric subgroup of $G(F)$ attached to a special vertex in the Bruhat-Tits building $\mathcal{B}_{\mathrm{red}}(G).$ This gives a bijection between the sets of isomorphism classes: 
\begin{gather*}
    \left\{
        \begin{tabular}{c}
            Irreducible admissible representations $(\sigma,V)$ of $G(F)$ \\
            such that $\sigma^{P_m} \ne 0$
        \end{tabular}
    \right\}
    \\ \Bigg\updownarrow \rlap{$\scriptstyle \Kaz_m$} \\
    \left\{
        \begin{tabular}{c}
            Irreducible admissible representations $(\sigma',V')$ of $G'(F')$ \\
            such that $\sigma'^{P_m'} \ne 0$
        \end{tabular}
    \right\}.
\end{gather*}
\end{enumerate}

The Deligne-Kazhdan correspondence has been used to transfer many results from characteristic $0$ to positive characteristic. Badulescu \cite{Badulescu02} used this correspondence to prove the Jacquet-Langlands local correspondence in positive characteristic. Ganapathy \cite{Ganapathy15} proved the LLC for $\GSp_4$ over local function fields, and Ganapathy and Varma \cite{GanapathyVarma17} established the LLC for split classical groups over local function fields, subject to certain restrictions on the characteristic. Furthermore, loosely speaking, the commutativity of the following diagram has been established for $\GL_n$ over local
    function fields of characteristic $p > 0$, and for $\Sp_{2n}$, split $\SO_{2n}$,
    $\SO_{2n+1}$, and $\GSp_4$ over local function fields of characteristic $p > 2$.
\[
\begin{tikzcd}[column sep=large, row sep=large]
    \left\{
        \begin{tabular}{c}
            Representations $\sigma$ of $G(F)$ with \\
            $\operatorname{depth}(\sigma) \le m$
        \end{tabular}
    \right\}
    \arrow[r, "\text{LLC}"]
    \arrow[d, "\text{Kazhdan}"']
    &
    \left\{
        \begin{tabular}{c}
            $\varphi : \operatorname{WD}_F \to {}^L G$ with \\
            $\operatorname{depth}(\varphi) \le m$
        \end{tabular}
    \right\}
    \arrow[d, "\text{Deligne}"]
    \\
    \left\{
        \begin{tabular}{c}
            Representations $\sigma'$ of $G'(F')$ with \\
            $\operatorname{depth}(\sigma') \le m$
        \end{tabular}
    \right\}
    \arrow[r, "\text{LLC}"]
    &
    \left\{
        \begin{tabular}{c}
            $\varphi' : \operatorname{WD}_{F'} \to {}^L {G'}$ with \\
            $\operatorname{depth}(\varphi') \le m$
        \end{tabular}
    \right\}
\end{tikzcd}
\]

Our main theorem is 
\begin{theorem}[Loose version]
    Let $F$ be a local field with $\Char(F) = p > 0.$ Then the formal degree conjecture holds for discrete series representations of following split groups defined over $F$\begin{itemize}
        \item $\GL_n,$
        \item $\GSp_4 (p > 2),$
        \item $\Sp_{2n}, \SO_{2n},$ and $\SO_{2n+1} (p >2).$
    \end{itemize}
\end{theorem}

For an exact version of the above theorem, see Theorem \ref{thm:Formal_deg_over_function_field}. To prove this, we first prove that the formal degrees of discrete series representations are equal under Kazhdan's isomorphism whenever the fields are sufficiently close (Theorem \ref{formal_degrees_are_equal1}). We also match the component groups and adjoint gamma factors over close local fields using Deligne's isomorphism (Section \ref{sec:Parameters_over_close_fields}). As a consequence of this, we prove the formal degree conjecture for above listed groups over local function fields (Theorem \ref{thm:Formal_deg_over_function_field}).

\section{Notation and preliminaries}\label{sec:Notation_and_prel}
\subsection{Fields}

\begin{itemize}
    \item Let $F$ be a non-Archimedean local field with ring of integers $\mathfrak{o}_F$, maximal ideal $\mathfrak{p}_F$, and uniformizer $\varpi_F$. We set $\mathfrak{f}_m = \mathfrak{o}_F /\mathfrak{p}_F^m$ and $\mathfrak{f} = \mathfrak{f}_1$. The size of $\mathfrak{f}$ is denoted by $q$.
    \item We fix an algebraic closure $\bar F$ of $F$. Let $F_s$ be the separable closure of $F$ in $\bar F$, and let $\kur$ be the maximal unramified extension of $F$ in $F_s$.
    \item Let $\Gamma_F = \Gal(F_s/F)$. Let $\Fr$ denote the geometric Frobenius element in $\Gal(F_{\unr}/F)$. We denote a lift of $\Fr$ to $\Gal(F_s/F)$ by the same symbol. Let $W_F$ be the Weil group and $I_F$ the inertia subgroup. We define $\WD_F := W_F \times \operatorname{SL}_2(\mathbb C)$.

   Let $I_F^r$ and $\Gamma_F^r$ be the $r$-th ramification subgroups of $I_F$ and $\Gamma_F$ in the upper numbering, respectively.
    
    \item We normalize the absolute value $\lvert \cdot \rvert$ on $F$ such that $\lvert \varpi_F \rvert = q^{-1}$. The valuation homomorphism $F^{\times} \to \mathbb Z$ is denoted by $\val_F$.
    \item Throughout the article, we fix an additive character $\psi$ of $F.$
\end{itemize}
Let $F'$ be another non-Archimedean local field. For any object $X$ associated with $F$, we denote by $X'$ the corresponding object associated with $F'$.   

\subsection{Groups}

\begin{itemize}
   \item Let $R$ be a ring, $\mathcal{X}$ an $R$-group scheme, and $R'$ an $R$-algebra. We denote by $\mathcal{X}(R')$ the set of morphisms $\operatorname{Spec}(R') \to \mathcal{X}$ over $\operatorname{Spec}(R)$. We define the base change $\mathcal{X}_{R'} := \mathcal{X} \times_{\operatorname{Spec}(R)} \operatorname{Spec}(R')$. If $R$ is a local ring, the special fiber of $\mathcal{X}$ is denoted by $\overline{\mathcal{X}}$. 
    \item For the remainder of the article, unless stated otherwise, we let $G$ denote a connected, reductive, quasi-split $F$-group. Let $T$ be a maximal $F$-torus of $G$ and $S$ be the maximal $F$-split torus of $T.$ Let $Z$ denote the maximal split $F$-torus inside the center of $G.$ 
    \item For any $F$-torus $T^{\dagger},$ let $V(T^{\dagger}) = X_*(T^{\dagger}) \otimes_{\mathbb Z} \mathbb R$. Let ${T^{\dagger}}^{\der}$ be the intersection of $T^{\dagger}$ with the derived subgroup of $G.$
    \item  Let $G(F)^0$ be the kernel of the Kottwitz homomorphism  \cite[Section 11.5]{KalethaPrasad23}. For an elementary definition of $G(F)^0$ see  \cite[Definition 2.5.13]{KalethaPrasad23}.
    
\end{itemize}

\subsection{Deligne's theory}{\label{subsec:Deligne's}}
We summarize relevant concepts from Deligne's work \cite{Deligne84}.
\begin{definition}\label{def:close_field}
Let $m \ge 1.$ The fields $F$ and $F'$ are said to be $m$-close if there is a ring isomorphism $ \mathfrak{o}_F/\mathfrak p_F^m \simeq \mathfrak{o}_{F'} / \mathfrak p_{F'}^m$.
\end{definition}
When $F$ and $F'$ satisfy the above definition, then we write $F \sim_m F'.$
Deligne \cite{Deligne84} considered the triplet
\[
\mathrm{Tr}_m(F) = (\mathfrak{o}_F / \mathfrak{p}_F^m,\; \mathfrak{p}_F / \mathfrak{p}_F^{m+1},\; \varepsilon),
\]
where $\varepsilon$ is the natural projection of $\mathfrak{p}_F / \mathfrak{p}_F^{m+1}$ onto 
$\mathfrak{p}_F / \mathfrak{p}_F^m$, and proved that $\Gamma_F / I_F^m$, together with its upper 
numbering filtration, is canonically determined by $\mathrm{Tr}_m(F)$. Hence an isomorphism of 
triplets $\theta_m : \mathrm{Tr}_m(F) \to \mathrm{Tr}_m(F')$ gives rise to an isomorphism
\begin{equation} \label{Eq:Del_m}
    \Gamma_F / I_F^m \xrightarrow{\;\mathrm{Del}_m\;} \Gamma_{F'} / I_{F'}^m,
\end{equation}
that is unique up to inner automorphisms. Further, we have, \begin{equation}\label{Eq:Del_m_preserves_upper_num}
    \Del_m(\Gamma_F^r/I_F^m) = \Gamma_{F'}^r/I_{F'}^m
\end{equation} for all $r \le m$. The equation~\ref{Eq:Del_m} holds true if we replace absolute Galois groups by Weil groups. The isomorphism $W_F/I_F^m \cong W_{F'}/I_{F'}^m$ will be again denoted by $\Del_m$.

More precisely, given an integer $f \ge 0$, let $\mathrm{ext}(F)^f$ be the category of finite separable
extensions $E/F$ satisfying the following condition: the normal closure $E_1$ of $E$ in $F_s$ satisfies
$\mathrm{Gal}(E_1/F)^f = 1$. Deligne proved that an isomorphism $\theta_m : \mathrm{Tr}_m(F) \to 
\mathrm{Tr}_m(F')$ induces an equivalence of categories $\mathrm{ext}(F)^m \to \mathrm{ext}(F')^m$.

Let $F\sim_m F'.$ Consider additive characters $\psi$ and $\psi'$ on $F$ and $F'$ respectively such that \cite[Section 3.7]{Deligne84}
\begin{itemize}
    \item Both $\psi$ and $\psi'$ have equal conductor $k$,
    \item $\psi \vert_{\mathfrak{p}_F^{k-m}/{\mathfrak{p}^k_F}}=\psi' \vert_{\mathfrak{p}_{F'}^{k-m}/{\mathfrak{p}^k_{F'}}}.$ 
\end{itemize} 
If the above conditions are satisfied, we write $\psi \sim_m \psi'.$

\subsection{Review of Bruhat-Tits buildings}
Let $\Bred(G)$ denote the reduced Bruhat--Tits building~\cite{BruhatTits84}. The group $G(F)$ acts by polysimplicial automorphisms on $\Bred(G)$. In this article, we mostly rely on~\cite{KalethaPrasad23} for the background material on Bruhat--Tits theory. By Axiom 4.1.1 of \cite{KalethaPrasad23}, any surjective homomorphism $G \to H$ with a central kernel induces an equivariant isomorphism $\Bred(G) \to \Bred(H)$. Let $\mathfrak{v}$ be a vertex in $\Bred(G)$. The stabilizer $P_{\mathfrak{v}} = G(F)^0_{\mathfrak v}$ of $\mathfrak{v}$ in $G(F)^0$ is called the parahoric subgroup of $G(F)$ associated to $\mathfrak{v}$. The given definition of a parahoric subgroup follows~\cite{HainesRapoport2008Parahoric}; although it differs from the definition given by Bruhat--Tits, the two are equivalent. The apartment $\mathcal A^{\Red}(S,F)$ is an affine space under $V(S^{\der})$. Let $\mathcal A(S,F) = \mathcal{A}^{\Red}(S,F) \times V(Z)$ be the enlarged apartment over $V(S) = V(S^{\der}) \oplus V(Z)$. By choosing a vertex $\mathfrak{v} \in \mathcal A^{\Red}(S,F)$ as the origin, we may identify $\mathcal A(S,F) \cong V(S)$. 
There exists a connected, affine $\mathfrak{o}_F$-group scheme $\mathcal G_{\mathfrak{v}}$ whose generic fiber is $G$ and $\mathcal G_{\mathfrak{v}}(\mathfrak{o}_F) = G(F)_{\mathfrak{v}}^0$. The group scheme $\mathcal G_{\mathfrak v}$ is called the Bruhat--Tits parahoric group scheme.
\subsection{Review of the motive and Haar measures}{\label{subsec:Haar_Measure}}

In this section, we review essential parts of \cite{Gross97} and \cite{GrossGan99}.
Let $W = N_{G_{F_s}}(T(F_s))/T(F_s)$ be the absolute Weyl group of $T$ in $G$.
The group $W$ acts on the rational vector space $E = X^*(T) \otimes \mathbb Q$ and thus on $\Sym^*(E)$. Let $R = \Sym^*(E)^W$, let $R_{+}$ be the ideal of elements of positive degree in $R$, and define the graded $\mathbb Q[\Gamma_F]$-module: $$V = R_{+}/R^2_{+} = \oplus_{d \ge 1}V_d.$$
Following \cite{Gross97}, the motive $M_G$ of $G$ is defined as the Artin--Tate motive $$ M_G = \oplus_{d \ge 1 } V_d(1-d)$$ over $F$. The twisted dual $M_G^{\vee}(1)$ of the motive $M_G$ is defined by $$M_G^{\vee}(1) = \oplus_{d \ge 1 } V_d(d).$$

Let us describe the Haar measure $|\omega_G|$ on $G(F)$ defined in \cite[Section 4]{Gross97}, which depends on the special vertex $\mathfrak{v}$ of the Bruhat--Tits building $\Bred(G)$ of $G$ given there. Let $\mathcal G = \mathcal G_{\mathfrak v}$ be the Bruhat--Tits group scheme over $\mathfrak{o}_F$ corresponding to the chosen special vertex $\mathfrak v$ with generic fiber $G$. Let $\overline{\mathcal G}^{\Red}$ denote the reductive quotient of the special fiber $\overline {\mathcal {G}}$ of $\mathcal G$. Let $\omega_G$ be a differential of top degree on $G$ over $F$ which has good reduction $(\bmod \varpi_F)$, and generates the $\mathfrak{o}_F$-submodule $\Hom(\wedge^{\Top} \Lie(\mathcal G), \mathfrak{o}_F)$ of $\Hom_{F}(\wedge^{\Top} \Lie(G),F)$. Then $\omega_G$ is determined up to multiplication by an element of $\mathfrak{o}_F^{\times}$. 
Let $|\omega_G|$ be the Haar measure on $G(F)$ which corresponds to the differential $\omega_G$. The measure $|\omega_G|$ can be defined in another way as in \cite[Section 5]{GrossGan99}. The equality of both measures was shown in \cite[Proposition 4]{Eugene04}.

The $L$-functions $L(M_G), L(M_G^{\vee}(1))$ are given \cite[ Section 5]{Gross97} by the formula $$L(M_G) = \det(1 - \Fr \vert M_G^{I_F})^{-1},$$ and $$L(M_G^{\vee}(1)) = \det(1 - \Fr \vert M_G^{\vee}(1)^{I_F})^{-1}.$$

Let $U$ be a continuous finite dimensional complex representation of $\Gamma_F$. We define the Artin conductor $a(U)$ as follows \cite[Section 4]{GrossGan99}. Let $L$ be the fixed field of $\ker(\Gamma_F \to GL(U)).$ Let $\Delta = \Gal(L/F)$. Let $$\Delta \supset \Delta_0 \supset \cdots $$ be the lower numbered ramification filtration of $\Delta.$ The group $\Delta_0$ is the inertia subgroup of $\Delta.$ If $r \ge -1$ is a real number then we set $\Delta_r := \Delta_{\lceil r \rceil}$, where $\lceil r \rceil $ is the smallest integer greater than or equal to $r$. Then $$a(U) = \sum_{i = 0}^{\infty}\frac{\dim(U/U^{\Delta_i})}{[\Delta_0: \Delta_i]}.$$ We write an alternate expression for $a(U)$ in terms of upper numbered filtration ramification groups $\Delta^r$. First note that \begin{equation}\label{Eq:artin_conductor_in_lower_filtration}
    a(U) = \int_{-1}^{\infty} \frac{\codim(U^{\Delta_r})}{[\Delta_0: \Delta_r]}\,dr.
\end{equation} Let $\varphi_{L/F}(x) = \int_0 ^x \frac{1}{[\Delta_0: \Delta_r]}dr$ be the Herbrand's function (Section 3, Chapter 4 of \cite{SerreLocalFields}), with the convention that $\varphi_{L/F}(x) = x$ for $x \in [-1,0].$ 
By the change of variable of $s = \varphi_{L/F}(r)$ in the equation \ref{Eq:artin_conductor_in_lower_filtration}, we get 
\begin{equation}\label{Eq: Artin_conductor_in_upper_filtration}
    a(U) = \int_{-1}^{\infty} \codim(U^{\Delta^s})ds.
\end{equation}

\begin{remark}\label{Remark: Artin_conductor_of_Weil_group}
    If $U$ is a finite dimensional, continuous representations of the Weil group $W_F,$ then the definition of Artin conductor of $U$ is similar (\cite{GrossReeder10}, section 2) and can be written as $$a(U) = \codim(U^{I_F}) + \int_0^{\infty} \codim(U^{I_F^s})ds.$$
\end{remark}

The Artin conductor of the motive $M_G$ of $G$ by the formula \cite[Section 4]{GrossGan99}: $$a(M_G) = \sum_{d \ge 1} (2d-1) a(V_d).$$

Let $\mu_{\psi}$ be the self-dual Haar measure on $F$ with respect to $\psi$. This normalization of the Haar measure determines the measure $|\omega_G|.$ Let $\mu_{Z \backslash G} := \mu_{Z(F) \backslash G(F), \psi}$ be the measure defined in \cite[p. 1221]{HiragaIchinoIkeda08_B}. This satisfies $$\mu_{Z \backslash G} = q^{-a(M_{G/Z})/2} \vert \omega_{Z \backslash G} \vert.$$ 
Let $P = \mathcal G(\mathfrak{o}_F)$ denote the parahoric subgroup of $G$ at $\mathfrak{v}$.

 Let us state Proposition 2.5 of \cite{Oesterle84}.
\begin{proposition}\label{Prop:Oesterle}
Let $\mathcal{H}$ be a separated, smooth group scheme of finite type and relative dimension $d = \dim(\mathcal H)$ over $\mathfrak{o}_F$. Let $\omega$ be a left-invariant differential form of maximum degree on the generic fiber $H$ of $\mathcal H$. Let $n$ be the integer characterized by $\omega_{e}(\wedge^d(\mathrm{Lie}(\mathcal H))) = \mathfrak{p}_F^n$, where $\omega_{e}$ denotes the value of $\omega$ at the identity $e \in H$. Then 
$$\int_{\mathcal H(\mathfrak{o}_F)} \lvert \omega \rvert = \mu_{\psi}(\mathfrak{o}_F)^d q^{-n-d} \# \mathcal H(\mathfrak{f}).$$
\end{proposition}

\begin{proposition}
    Let $\omega_G$ be as defined in Subsection \ref{subsec:Haar_Measure}.
    We have $$\int_{P} |\omega_G| =  \mu_{\psi}(\mathfrak{o}_{F})^{\dim \mathcal G} L(M_G^{\vee}(1))^{-1}. $$
\end{proposition}
 
\begin{proof}
   This is Proposition 4.7 of \cite{Gross97} for fields of characteristic zero. This also holds in non-zero characteristics. 
Let $M = \wedge^{\Top} \Lie(\mathcal G)$ and $M^{\vee} = \Hom_{\mathfrak{o}_F}(\wedge^{\Top} \Lie(\mathcal G), \mathfrak{o}_F)$. Then both $M$ and $M^{\vee}$ are one-dimensional $\mathfrak{o}_F$-modules. Let $\alpha$ be an $\mathfrak{o}_F$-basis of $M$ and let $\alpha^{\vee}$ be the corresponding dual basis of $M^{\vee}$. Thus, there exists $c \in \mathfrak{o}_F$ such that $\omega_{G,e} = c \alpha^{\vee}$. The $\mathfrak{o}_F$-submodule generated by $\omega_{G,e}$ is $c M^{\vee}$. 
But by definition, $\omega_{G,e}$ generates the $\mathfrak{o}_F$-submodule $M^{\vee}$ of $\Hom_F(\wedge^{\Top}\Lie(G), F)$. Thus, $M^{\vee} = c M^{\vee}$. This implies that $c \in \mathfrak{o}_F^{\times}$. 
So we have $\omega_{G,e}(\wedge^{\Top} \Lie(\mathcal G)) = \mathfrak{o}_F = \mathfrak{p}_F^0$. This implies that $n=0$ in the hypothesis of Proposition~\ref{Prop:Oesterle}. Then it follows from Proposition \ref{Prop:Oesterle} that \[
\int_{P} |\omega_G|
= \mu_{\psi}(\mathfrak{o}_F)^{\dim \mathcal G}
\frac{\#\overline{\mathcal G}(\mathfrak f)}{q^{\dim \mathcal G}}.
\]    Note that \begin{equation}\label{Eq:reductive_quotient_size}
        \frac{\# \overline{\mathcal{G}}(\mathfrak{f})}{q^{\dim \overline{\mathcal{G}}}} = \frac{\# \overline{\mathcal{G}}^{\mathrm{red}}(\mathfrak{f})}{q^{\dim \overline{\mathcal{G}}^{\mathrm{red}}}}.
    \end{equation}
    The validity of Equation \ref{Eq:reductive_quotient_size} can be seen as follows:
    \begin{itemize}
        
       \item Corollary 14.2.7 of \cite{Springer98} implies $\# R_u(\overline{\mathcal G})(\mathfrak{f}) = q ^{\dim(R_u(\overline{\mathcal G}))}$.
        \item By Lang's theorem \cite[Corollary 17.98]{Milne17}, $H^1(\mathbb F_q, R_u(\overline{\mathcal G}))=1$.
        \item $\dim(\overline{\mathcal G}) = \dim(R_u(\overline{\mathcal G})) + \dim(\overline{\mathcal G}^{\Red})$.
        
    \end{itemize}
   Lang's theorem implies that the sequence $$1 \to R_u(\overline{\mathcal G})(\mathfrak{f}) \to  \overline{\mathcal{G}}(\mathfrak{f}) \to \overline{\mathcal G}^{\Red} (\mathfrak{f}) \to 1$$ is exact. Thus, $$ \frac{\# \overline{\mathcal{G}}(\mathfrak{f})}{q^{\dim \overline{\mathcal{G}}}} = \frac{\#R_u(\overline{\mathcal G})(\mathfrak{f}) \#\overline{\mathcal G}^{\Red}(\mathfrak{f})}{q^{\dim{\overline{\mathcal G}^{\Red}}}q^{\dim(R_u(\overline{\mathcal G}))}} = \frac{\# \overline{\mathcal{G}}^{\mathrm{red}}(\mathfrak{f})}{q^{\dim \overline{\mathcal{G}}^{\mathrm{red}}}}.$$

    By Proposition 4.5 of \cite{Gross97}, we know that the motive of $\overline {\mathcal G}^{\Red}$ over $\mathfrak{f}$ is $M_G^{I_F}.$ By Steinberg's formula \cite[Section 3.1]{Gross97}, one concludes that $$\frac{\# \overline{\mathcal{G}}^{\mathrm{red}}(\mathfrak{f})}{q^{\dim \overline{\mathcal{G}}^{\mathrm{red}}}} = \det(1 - \Fr \vert M_G^{\vee}(1)^{I_F}) = L(M_G^{\vee}(1))^{-1}.$$
    
\end{proof}

Let $P_m = \ker(\mathcal G(\mathfrak{o}_F) \to \mathcal G(\mathfrak f_m)).$ Write the coset decomposition $P = \sqcup_{i} g_iP_m.$ Then $$\int_P |\omega_G| = \sum_{i} \int_{g_iP_m} |\omega_G| = [P:P_m] \int_{P_m} |\omega_G| = \# \mathcal G(\mathfrak f_m) \int_{P_m} |\omega_G|.$$ We note the following corollary:

\begin{corollary}\label{Volume_of_congruence_subgroup}
    We have $$\int_{P_m} |\omega_G| =  \mu_{\psi}(\mathfrak{o}_F)^{\dim \mathcal G} \frac{ L(M_G^{\vee}(1))^{-1}}{\# \mathcal G(\mathfrak f_m)}. $$
\end{corollary}

\subsection{Summary of Sections 3 and 4 of \cite{Ganapathy19}}\label{subsec:Gan19}
We summarize the required results from these sections.

Let $(G_0, T_0, B_0, \{u_\alpha\}_{\alpha\in \Tilde{\Delta}})$ be a pinned, split, connected, reductive $\mathbb{Z}$-group with based root datum $(R,\Tilde{\Delta})$. 
We know that the $F$-isomorphism classes of quasi-split groups $G$ that are $F$-forms of $G_0$ 
are parametrized by the pointed cohomology set $H^1(\Gamma_F, \mathrm{Aut}(R,\Tilde{\Delta}))$.
Let $E_{\mathrm{qs}}(F,G_0)_m$ be the set of $F$-isomorphism classes of quasi-split groups $G$ that split 
(and become isomorphic to $G_0$) over an at most $m$-ramified extension of $F$. 
This is parametrized by the cohomology set 
$H^1(\Gamma_F / I_F^m, \mathrm{Aut}(R,\Tilde{\Delta}))$. 
By the Deligne isomorphism, there is a bijection \cite[Lemma 3.3]{Ganapathy19}
\[
\mathcal C_m :E_{\mathrm{qs}}(F,G_0)_m \longrightarrow E_{\mathrm{qs}}(F',G_0)_m,\qquad 
\mathcal C_m(G) =  G',
\] provided $F$ and $F'$ are $m$-close.   

Let $H = H(G_0,T_0,B_0, \{ u_{\alpha}\}_{\alpha \in \Tilde{\Delta} })$ be the subgroup of $\Aut(G_{0,F_s})$ consisting of all $a$ such that $a(B_0)= B_0, a(T_0) = T_0$ and $\{a \circ u_{\alpha} : \alpha \in \Tilde{\Delta}\}=\{u_{\alpha}: \alpha \in \Tilde{\Delta}\}.$ Let $\Omega = (F_s)^{I_F^m}.$ Let $(G, \phi)$ be a pair consisting of a quasi-split connected reductive group $G$ over $F$ and an $\Omega$-isomorphism $\phi : G_0 \times_{\mathbb Z} \Omega \to G \times_{F} \Omega$ such that the Galois action on $G(F_s)$ is given by a fixed element $s\in Z^1(\Gal(\Omega/F), \Aut(R, \Tilde {\Delta})) \cong Z^1(\Gal(\Omega/F),H).$ We may and do assume that there is a finite Galois, at most $m$-ramified extension $K$ of $F$ over which $\phi$ is defined, that is, $s \in Z^1(\Gal(K/F),H).$ An element $\gamma \in \Gal(K/F)$ acts on $G(K)$ by $\gamma \cdot \phi(x) = \phi(s(\gamma) \cdot (\gamma \cdot x)),$ where $x \in G_0(K).$ Then $T = \phi(T_0)$ is a maximal torus of $G$ defined over $F$, and $B = \phi(B_0)$ is a Borel subgroup of $G$ containing $T$ and defined over $F.$ Using Lemma 3.3 of \cite{Ganapathy19}, we can choose corresponding objects $(G', \phi'), T', B', \gamma' = \Del_m(\gamma), s', K'/F'$ over $F'$ such that the $F$-quasi-split data $(G,B,T)$ with cocycle $s$ corresponds to the $F'$-quasi-split data $(G',B',T')$ with cocycle $s'$, and there are $\Del_m$-equivariant isomorphisms $X_*(T) \to X_*(T')$ and $X^*(T) \to X^*(T')$ \cite[Lemma 3.4]{Ganapathy19}. Also, there exists a $\Del_m$-equivariant isomorphism $\Phi_m : \Phi(G,T) \to \Phi(G',T').$ Note that this isomorphism restricts to an isomorphism $\Phi(G,S) \to \Phi(G',S')$ of relative root systems, which we also denote by $\Phi_m.$ Since $X_*(Z) = \{ \lambda \in X_*(S) : \langle \alpha, \lambda \rangle =0 \quad \forall \alpha \in \Phi(G,S) \},$ the isomorphism $\Phi_m$ induces an isomorphism $X_*(Z) \simeq X_*(Z').$ 
 
It has been shown in \cite[Proposition 4.4]{Ganapathy19} that there is a simplicial isomorphism $\mathcal A_m : \mathcal A^{\Red}(S,F) \to \mathcal A^{\Red}(S',F').$ Moreover, by \cite[Theorem 4.5]{Ganapathy19}, there exists an integer $l \ge m$ such that if $F$ and $F'$ are $l$-close, we obtain an isomorphism 
\[
    p_m : P/P_m \times P/P_m \to P'/P_m' \times P'/P_m'
\]
and an isomorphism $\mathcal{G}(\mathfrak{f}_m) \to \mathcal{G}'(\mathfrak{f}_m')$. Here, $\mathcal{G}$ and $\mathcal{G}'$ denote the Bruhat-Tits group schemes corresponding to the special vertices $\mathfrak{v} \in \mathcal{A}^{\Red}(S,F)$ and $\mathfrak{v}' = \mathcal{A}_m(\mathfrak{v})$, respectively.

\subsection{Parameters}\label{sec:Langland_parameters}
The dual group $\hat G$ of $G$ is a connected, reductive group over $\mathbb{C}$ with based root datum $\mathcal{R}(\hat G)=(X_*(T), X^*(T), \hat \Delta, \Delta)$ dual to that of $G$, where $\Delta = \Delta(G,T)$. Fix a pinning $\mathcal{E} = (\hat G,\hat B,\hat T, \{u_{\alpha}\})$ of $\hat G$. The Galois group $\Gamma_F$ acts by automorphisms of $\hat G$ that preserve the pinning $\mathcal{E}$. There is a canonical isomorphism between the group $\Aut(\hat G, \mathcal{E})$ of such automorphisms and the automorphism group $\Aut(\mathcal{R}(\hat G))$ of the based root datum. Thus, there is a natural map $\Gamma_F \to \Aut(\hat G, \mathcal{E}) \cong \Aut(\mathcal{R}(\hat G))$. The $L$-group of $G$ is defined as the semidirect product
\[ {}^L G = \Gamma_F \ltimes\hat{G} .\]

A parameter for $G$ is a homomorphism $\varphi : \WD_F \to {}^L G$ such that: 
\begin{itemize}
    \item $\varphi|_{\SL_2} : \SL_2 \to \hat G$ is a homomorphism of algebraic groups over $\mathbb C;$
    \item $\varphi$ is continuous on $I_F$ and $\varphi(\Fr)$ is semisimple;
    \item the composition $W_F \xrightarrow{\varphi} {}^L G \to \Gamma_F$ is the canonical inclusion $W_F \to \Gamma_F.$
\end{itemize}
Two parameters are said to be equivalent if they are conjugate under $\hat G.$ Let $\mathcal L(G,F)$ denote the set of $\hat G$-conjugacy classes of parameters. For a parameter $\varphi,$ define 
\[
    \depth(\varphi) = \inf_{r} \{r : \varphi|_{I_F^{r+}} = 1 \}.
\]
Let $\mathcal L_m(G,F) \subset \mathcal L(G,F)$ be the collection of parameters whose depth is less than or equal to $m.$

Let $S_{\varphi}$ denote the centralizer in $\hat G$ of the image $\varphi(\WD_F)$, and let $\overline{S_{\varphi}} = S_{\varphi}/Z(\hat G)^{\Gamma_F}.$ We say that $\varphi$ is (essentially) discrete if $\overline{S_{\varphi}}$ is finite. Let $\hat{G}^{\sharp}$ be the dual group of $G/Z,$ which is a subgroup of $\hat G$ containing the derived subgroup of $\hat G.$ Let $S_{\varphi}^{\sharp}$ be the centralizer of $\varphi(\WD_F)$ in $\hat{G}^{\sharp}$, and let $\mathcal{S}^{\sharp}_{\varphi} = \pi_0(S^{\sharp}_{\varphi}).$

Fix a finite-dimensional complex representation $r : {}^L G \to \GL(V).$ Then $r \circ \varphi : \WD_F \to \GL(V)$ is a finite-dimensional representation of $\WD_F$, which we again denote by $\varphi.$ Since $V$ is semisimple under $\varphi,$ we have a canonical decomposition
\[
    V = \bigoplus_{n \ge 0} V_n \otimes \Sym^n,
\]
where $\Sym^n = \Sym^n(\mathbb C^2)$ is the irreducible $\SL_2$-representation of dimension $n+1,$ and $V_n$ is a semisimple complex representation of $W_F.$ Then the local $L$-functions and local epsilon factors are given in \cite[Section~2]{GrossReeder10} by 
\begin{align*}
    L( \varphi,V,s) &= \prod_{n \ge 0} \det(1- q^{-s - (n/2)} \varphi(\Fr)|_{V_n^{I_F}})^{-1}, \\
    \epsilon(\varphi, V,s) &= \omega(\varphi,V)q^{\alpha(\varphi,V)(-s + (1/2))},
\end{align*}
where 
\[
    \alpha(\varphi, V) = \sum_{n \geq 0} (n+1)a(V_n) + \sum_{n \geq 1} n \cdot \dim V_n^{I_F}
\]
and 
\[
    \omega(\varphi, V) = \prod_{n \geq 0} w(V_n)^{n+1} \cdot \prod_{n \geq 1} \det(-\varphi(\Fr)|_{V_n^{I_F}})^n,
\]
where $a(V_n)$ and $w(V_n)$ are the Artin conductors and local constants of the representations $V_n$ of $W_F$. The gamma factor of $(\varphi,V)$ is defined by  
\begin{equation}\label{Eq:adjoint_gamma_factor}
    \gamma(\varphi, V,s) = \frac{L(\varphi, V^{*},1-s) \epsilon(\varphi,V,s)}{L(\varphi, V,s)},
\end{equation} 
where $V^*$ is the representation of $\WD_F$ dual to $V.$

\subsection{The formal degree conjecture}

Let $\Pi(G,F)$ denote the set of equivalence classes of irreducible, admissible representations of $G(F).$ The local Langlands conjecture predicts a surjective map $\LLC : \Pi(G,F) \to \mathcal L(G,F), \pi \mapsto \varphi_{\pi}$ with various properties. See \cite[Section 6]{taibi2025locallanglandsconjecture}. The inverse image $\LLC^{-1}(\varphi)$ is denoted by $\Pi_{\varphi}(G,F).$ Assuming the LLC for quasi-split groups, we have a bijection 
\[
    \Pi_{\varphi}(G,F) \to \Irr(\pi_0(\overline{S_{\varphi}})), \quad \pi \mapsto \rho_{\pi}.
\]

Let $V = \Lie(\hat G)/\Lie(Z(\hat G)^{\Gamma_F}).$ Define $\gamma(\varphi) := \gamma(\varphi, V,0).$ Let $\Pi^2_{\varphi}(G,F) \subset \Pi_{\varphi}(G,F)$ be the set of discrete series representations. The formal degree conjecture states: 

\begin{conjecture}[\cite{HiragaIchinoKaoru08}]\label{formal_degree_conjecture_stmnt}
    Let $\varphi$ be a discrete parameter, and let $\pi \in \Pi_{\varphi}^2(G,F).$ Then 
    \[
        d(\pi, \mu_{Z\backslash G}) = \frac{\dim \rho_{\pi}}{\#\mathcal{S}_{\varphi}^\sharp} \lvert \gamma(\varphi) \rvert.
    \]
\end{conjecture}
If this conjecture is true, then we say that the formal degree conjecture holds for $(F,G,\varphi, \pi,\psi).$

\begin{remark}\label{Rem: component_grp_abelian}
    If the component group $\pi_0(\overline{S_{\varphi}})$ is abelian, then Conjecture~\ref{formal_degree_conjecture_stmnt} reduces to proving
    \[
        d(\pi, \mu_{Z\backslash G}) = \frac{1}{\#\mathcal{S}_{\varphi}^\sharp} \lvert \gamma(\varphi) \rvert.
    \]
\end{remark}

\subsection{Cartan decomposition}
Since $G$ is quasi-split, we have $T = C_G(S).$ 
Let $\Omega_T := (X_*(T)_{I_F})^{\Fr}.$ Also set $\overline{\Omega}_T := \Omega_T/(\Omega_T)_{\tor}.$

There is an isomorphism $\mathcal{A}(S,F) \cong \overline{\Omega}_T \otimes_{\mathbb Z} \mathbb R$~\cite[Section 3.3]{Rostami15}. It follows that there exists a tuple $\Psi_T = (\overline{\Omega}^{\vee}, \overline{\Sigma}, \overline{\Omega}_T, \overline{\Sigma}^{\vee}, \langle \cdot, \cdot \rangle )$ such that $\Psi_T$ is a reduced root datum~\cite[Lemma 3.3.1]{Rostami15}. By the above isomorphism, we can view $\mathfrak A \subset \mathcal C$ as a subset of $\overline{\Omega}_T \otimes_{\mathbb Z} \mathbb R$, where $\mathcal C$ is the Weyl chamber corresponding to $\Delta(G,S)$ in $\mathcal{A}^{\Red}(S,F)$, and $\mathfrak{A}$ is the alcove whose closure contains the special vertex $\mathfrak{v}$. An element $\lambda \in \overline{\Omega}_T$ is said to be dominant if and only if $\langle \lambda, \alpha \rangle \ge 0$ for all $\alpha \in \Delta(B,S)$. Let $\overline{\Omega}_{T,+}$ denote the subset of all dominant elements in $\overline{\Omega}_T$. Let $\Omega_{T,+}$ denote the preimage of $\overline{\Omega}_{T,+}$. An equivalent description of $\overline{\Omega}_{T,+}$ is given in Section~6 of~\cite{HenniartVigneras15}. Let $P$ be a parahoric subgroup of $G$ corresponding to a special vertex. We note the following result:
\begin{theorem}[{\cite[Theorem 1.0.3]{HainesRostami10}}]
    We have a bijection $P \backslash G(F) / P \cong W/ \Omega_T$.
\end{theorem}

Let $\kappa_T : T(F) \to \Omega_T$ denote the surjective homomorphism, and let $\overline{\kappa}_T : T(F) \to \overline{\Omega}_T$ be the composition of $\kappa_T$ with the natural projection map. Let $p : \overline{\Omega}_T \to T(F)$ be the group-theoretic section of $\overline{\kappa}_T$ defined in~\cite{Ganapathy24}. Let $s : (\Omega_T)_{\mathrm{tor}} \to T(F)$ be a set-theoretic section of $\kappa_T$. As $\Omega_T$ is a finitely generated abelian group, we have $\Omega_T = \overline{\Omega}_T \oplus (\Omega_T)_{\mathrm{tor}}$. This defines a set-theoretic section $q : \Omega_T \to T(F)$ of the map $\kappa_T$ given by $q(\tau + \tau_1) = p(\tau) s(\tau_1)$, where $\tau \in \overline{\Omega}_T$ and $\tau_1 \in (\Omega_T)_{\mathrm{tor}}$. Also, note that $\Omega_{T,+} = \overline{\Omega}_{T,+} \oplus (\Omega_T)_{\mathrm{tor}}$. Define $n_{\tau} := q(\tau)$ for $\tau \in \Omega_T$.

The preceding theorem can be restated as follows:
\begin{theorem}[{\cite[Proposition 6.3]{HenniartVigneras15}}] \label{thm:Cartan_decomp}
    $$G(F) = \bigsqcup_{\tau \in  \Omega_{T,+}} P n_{\tau} P.$$
\end{theorem}

\subsection{Review of Kazhdan's isomorphism}{\label{Kazdhan}}

For general connected reductive groups, a version of Kazhdan's isomorphism was proved in \cite{Ganapathy22}. We will describe the isomorphism in the quasi-split case. For any positive integer $m$, if $F \sim_m F'$, let $G$ and $G'$ be as in Subsection \ref{subsec:Gan19}.
\begin{proposition}[\cite{Ganapathy22}, Proposition 4.3] Let $m$ be a positive integer.
    Suppose $F$ and $F'$ are $e$-close for some $e \gg m$; then we have the following isomorphisms,
    \begin{enumerate}
        
        \item $W(G,S) \cong W(G',S'),$
        \item $\Omega_T \cong \Omega_{T'},$
       \item $P \backslash G(F) / P \cong P' \backslash G'(F') / P'.$
    \end{enumerate}
    
\end{proposition}

By Equation (4-3) in Proposition 4.4 of \cite{Ganapathy22}, the following result holds:

\begin{proposition}\label{Double_coset_Vol_mathcing_prop}
Let $\tau' \in \Omega_{T',+}$ be the image of $\tau \in \Omega_{T,+}$ under the isomorphism $\Omega_{T} \cong \Omega_{T'}$ of part (2) above. We have
$$
\Vol(P_m n_{\tau}P_m; dg) = \Vol(P_m' n_{\tau'}P_m'; dg'), $$
where $dg$ and $dg'$ are Haar measures on $G$ and $G'$ such that $\Vol(P_m;dg) = 1=\Vol(P_m';dg').$

\end{proposition}

\begin{theorem}[\cite{Ganapathy22}]{\label{KazIsomorphism}}
    Let $m \ge 1.$ There exists an integer $l \ge m$ such that if $F$ and $F'$ are $l$-close then there is an isomorphism of Hecke algebras $$\Kaz_m : \mathcal H(G,P_m) \to \mathcal H(G', P_m').$$
\end{theorem}

For $\tau \in  \Omega_{T,+}$ let $G_{\tau}$ denote $Pn_{\tau}P.$ This set is a homogeneous space under $P \times P$ under the action $(k_1,k_2) \cdot g = k_1 g k_2^{-1}.$ Let $X$ denote the discrete set of $P_m$-double cosets of $P_m \backslash G(F) /P_m,$ and let $X_{\tau} \subset X$ denote the set of double cosets in $G_{\tau}.$ Then $X_{\tau}$ is a homogeneous space under the group $P/P_m \times P/P_m.$ Let $\Gamma_{\tau} \subset P/P_m \times P/P_m$ be the stabilizer of $P_m n_{\tau}P_m.$ Let $T_{\tau} \subset P \times P$ be a set of representatives of $ (P/P_m \times P/P_m)/\Gamma_{\tau}.$

\section{Functoriality of quasi-split groups over close local fields}
\subsection{Pinned groups and based root data}
Let $ \mathcal R = (X, \Phi, \Delta, X^{\vee}, \Phi^{\vee}, \Delta^{\vee})$ and $ \mathcal R^{\dagger} = (X^{\dagger}, \Phi^{\dagger}, \Delta^{\dagger}, {X^{\dagger}}^{\vee}, {\Phi^{\dagger}}^{\vee}, {\Delta^{\dagger}}^{\vee})$ be based root data. We say that a tuple $(g, \tau) : \mathcal R^{\dagger} \to \mathcal R$ is a morphism of based root data if 
\begin{itemize}
    \item $g: X^{\dagger} \to X$ is a $\mathbb Z$-linear map,
    \item $\tau : \Phi \to \Phi^{\dagger}$ is a bijection such that $$g(\tau(\alpha))  =\alpha, g^{\vee}(\alpha^{\vee}) = (\tau (\alpha))^{\vee},$$
    \item $\tau (\Delta) = \Delta^{\dagger}.$
\end{itemize}

Recall that an $F$-pinning of a quasi-split $F$-group $G$ is a quadruple $(G, B, T, \{x_{\alpha}\})$ consisting of a Borel $F$-subgroup $B$, a maximal $F$-torus $T \subset B$ and, for each $\alpha \in \Delta = \Delta(B, T)$ an isomorphism $x_{\alpha} : \mathbb G_a \to U_{\alpha}$ over $F_s$, where $U_{\alpha}$ is the root subgroup of $G_{F_s}$ corresponding to $\alpha$, such that the set $\{x_{\alpha}\}$ is invariant under the action of $\Gamma_F$. By a homomorphism $f$ of pinned groups $(G,B,T, \{ x_{\alpha} \}) \to (G^{\dagger}, B^{\dagger}, T^{\dagger}, \{x_{\alpha^{\dagger}} \})$, we mean an $F$-homomorphism $f : G\to G^{\dagger}$ that respects pinnings and satisfies the following conditions:
\begin{enumerate} \label{Condition:Homomorphism_pinned_group}

    \item The kernel $\ker f$ is a central subgroup of $G$, and $(G^{\dagger})_{\mathrm{der}} \subseteq \mathrm{im}(f)$.
    \item \label{item:df_s_non_Zero}  For every simple root $\alpha \in \Delta(B,T),$ the Lie algebra map $d f_s |_{\mathrm{Lie}(U_{\alpha})}$ is non-zero, where $f_s$ is the base change of $f$ from $F$ to $F_s$.
\end{enumerate}
In general, we have $f_s(x_{\alpha}(k)) = x_{\tau(\alpha)}(k^{q(\alpha)})$ for all $k \in \mathbb G_a(D)$ and $\alpha \in \Delta$, where $D$ is any $F_s$-algebra and $q(\alpha)$ is either $1$ or a power of $\Char(F)$.
Item~\ref{item:df_s_non_Zero} is equivalent to $q(\alpha) = 1$ for all $\alpha \in \Delta(B,T)$. 

Let $\mathcal R(G, B, T, \{x_{\alpha}\})$ be the based root datum corresponding to the pinning $(G, B, T, \{x_{\alpha}\}).$ Note that a homomorphism $f$ of pinned groups 
$$(G,B,T, \{ x_{\alpha} \}) \to (G^{\dagger}, B^{\dagger}, T^{\dagger}, \{x_{\alpha^{\dagger}} \})$$ 
induces a morphism of based root data 
$$\mathcal R(f) : \mathcal R (G^{\dagger}, B^{\dagger}, T^{\dagger}, \{x_{\alpha^{\dagger}} \}) \to \mathcal R(G, B, T, \{x_{\alpha}\}).$$ In fact, we have

\begin{theorem}\label{thm: isotpy_theorem}
    Let $(G, B, T, \{x_{\alpha}\})$ and $(G^{\dagger}, B^{\dagger}, T^{\dagger}, \{x_{\alpha^{\dagger}}\})$ be pinned groups. Every morphism from the based root datum of $G^{\dagger}$ to that of $G$ arises from a unique homomorphism of pinned groups $G \to G^{\dagger}$.
\end{theorem} 

For split groups the above theorem is stated in \cite[Proposition 3.13 ]{JayTaylor19} and \cite[Proposition 1.14]{Janzten03}. For the quasi-split version, see \cite[Theorem 2.9.3]{KalethaPrasad23}.

\subsection{Functoriality over close local fields}
Let $F$ and $F'$ be such that $Tr_m(F) \cong Tr_m(F')$. 
Let $\Pin_{G_0} = (G_0, B_0, T_0, \{u_{\alpha}\})$ and $\Pin_{H_0} = (H_0, B_0^{\dagger}, T_0^{\dagger}, \{u_{\alpha^{\dagger}}\})$ be pinned, split, connected, reductive $\mathbb Z$-groups. Let $f: G_{0,F_s} \to H_{0,F_s}$ be a homomorphism of pinned groups that is not necessarily defined over $\mathbb Z$. Then, by Theorem \ref{thm: isotpy_theorem}, we get a unique homomorphism $f' : G_{0,F_s'} \to H_{0,F_s'}$ of pinned groups with $\mathcal R(f') = \mathcal R(f)$. Note that, for each $\sigma \in \Gamma_F$, the Galois action on $G_{0,F_s}$ and $H_{0,F_s}$ provides a homomorphism ${^{\sigma}f }:= \sigma \circ f \circ \sigma^{-1} : G_{0,F_s} \to H_{0,F_s}$.
For each $\sigma \in \Gamma_F$, let $\sigma' = \Del_m(\sigma)$. Note that the morphisms satisfy $\mathcal R(\sigma : G_{0,F_s} \to G_{0,F_s}) = \mathcal R(\sigma' : G_{0,F'_s} \to G_{0,F'_s})$ and $\mathcal R(\sigma : H_{0,F_s} \to H_{0,F_s}) = \mathcal R(\sigma' : H_{0,F'_s} \to H_{0,F'_s})$. So, $\mathcal R(^\sigma f) = \mathcal R(\sigma)\circ \mathcal R(f) \circ \mathcal R(\sigma)^{-1} = \mathcal R(^{\sigma'}f')$. By Theorem \ref{thm: isotpy_theorem} (uniqueness), $$(^{\sigma} f)' = {^{\sigma'} f'}.$$ Further, note that the correspondence $f \mapsto f'$ respects the composition of homomorphisms.

Recall that $E_{\qs}(F,G_0)_m = \{ [G] \in E(F,G_0)_m : s_G \in \im(q) \}$, where $E(F,G_0)$ is the set of $F$-isomorphism classes of connected, reductive $F$-groups $G$ with $G_{F_s} \cong G_{0,F_s}$, and $q : H^1(\Gamma_F, \Aut(R, \Tilde{\Delta})) \cong H^1(\Gamma_F, \Aut_{\pin}(G_{0,F_s})) \to H^1(\Gamma_F, \Aut(G_{0,F_s}))$ is the section of the natural map $H^1(\Gamma_F, \Aut(G_{0,F_s})) \to H^1(\Gamma_F, \Aut(R, \Tilde{\Delta}))$. Fix a pinned isomorphism $\alpha_G : G_{F_s} \to G_{0,F_s}$. Then, for all $\sigma \in \Gamma_F$, 
$$s_G(\sigma) = \alpha_G \circ {^{\sigma} \alpha_G^{-1}}.$$ 
Also, recall that the bijection 
$$\mathcal C_m:E_{\qs}(F,G_0)_m \to E_{\qs}(F',G_0)_m$$ 
is given by $G \mapsto G'$, where $s_{G'} = q'(t_{G'})$ and $t_{G'} = \mathfrak Q^c_m(t_G) = t_G \circ \Del_m^{-1}$ with $s_G = q(t_G)$, and where $\mathfrak{Q}^c_m : Z^1(\Gamma_F/I_F^m, \Aut(R, \Tilde{\Delta})) \cong Z^1(\Gamma_{F'}/I_{F'}^m, \Aut(R, \Tilde{\Delta}))$. Further, the bijection $\mathcal C_m$ gives rise to pinnings $\Pin_G$ of $G$ and $\Pin_{G'}$ of $G'$ such that there exists a $\Del_m$-equivariant isomorphism $\mathcal R(\Pin_G) \to \mathcal R(\Pin_{G'})$ of the corresponding based root data. By the definition of $t_{G'}$, we have $(t_G(\sigma))' = t_{G'}(\sigma')$, and thus $(s_G(\sigma))' = s_{G'}(\sigma')$.

Now, let $[H] \in E_{\qs}(F,H_0)_m$ with a pinning $\Pin_H$. Let $f : \Pin_G \to \Pin_H$ be a homomorphism of pinned groups. Let $f_s : G_{F_s} \to H_{F_s}$ be the base change map from $F$ to $F_s$. Let $f_0 : G_{0,F_s} \to H_{0,F_s}$ be the map defined by 
$$f_0 = \alpha_H \circ f_s \circ \alpha_G^{-1}.$$ 
Then note that $f_0$ is a homomorphism of pinned groups $G_{0,F_s} \to H_{0,F_s}$. Thus, we have a corresponding map $f_0' : G_{0,F_s'} \to H_{0,F_s'}$ of pinned groups. For each $\sigma \in \Gamma_F$, the following holds: 
\begin{equation}\label{eq:cocycle_condition}
    s_{H}(\sigma) \circ {^{\sigma}f_0} = f_0 \circ s_G(\sigma).
\end{equation}
Note that both functions on the right-hand side and the left-hand side of equation \ref{eq:cocycle_condition} are homomorphisms $G_{0,F_s} \to H_{0,F_s}$ of pinned groups. 
Thus, we have 
$$s_{H'}(\sigma') \circ {^{\sigma'} f_0'} = f_0' \circ s_{G'}(\sigma')$$ 
for all $\sigma' \in \Gamma_{F'}$. Let $\alpha_{G'} = (\alpha_G)'$ and $\alpha_{H'} = (\alpha_H)'$. Define $f_s' = \alpha_{H'}^{-1} \circ f_0' \circ \alpha_{G'} : G_{F_s'} \to H_{F_s'}$. Then, for each $\sigma' \in \Gamma_{F'}$, we have $^{\sigma'}f_s' = f_s'$. Thus, $f_s'$ descends to a unique map $f' : G' \to H'$ of pinned groups over $F'$.

Let $[\Pin_G]$ be the $F$-isomorphism class of a pinned quasi-split $F$-group $\Pin_G = (G,B,T,\{ x_{\alpha}\})$. Let $E_{\qs,\pin}(F)$ denote the collection of all such classes $[\Pin_G]$, where $G$ ranges over quasi-split, connected, reductive $F$-groups. Suppose $[\Pin_G] = [\Pin_{G_1}]$ and $[\Pin_{H}] = [\Pin_{H_1}]$. Let $f : \Pin_G \to \Pin_H$ and $f_1 : \Pin_{G_1} \to \Pin_{H_1}$ be homomorphisms of pinned groups. We say that $f \sim f_1$ if there exist unique isomorphisms $a : \Pin_G \to \Pin_{G_1}$ and $b : \Pin_H \to \Pin_{H_1}$ such that $f_1 = b \circ f \circ a^{-1}$. Indeed, by Theorem \ref{thm: isotpy_theorem}, isomorphisms $a$ and $b$ exist and are unique, making $\sim$ a well-defined equivalence relation. We denote the class of $f$ by $[f]$. So, we may regard $E_{\qs,\pin}(F)$ as a category where:
\begin{itemize}
    \item The objects are the $F$-isomorphism classes $[\Pin_G]$ of pinned, quasi-split $F$-groups $\Pin_G$.
    
    \item Given two objects $[\Pin_G]$ and $[\Pin_H]$, the morphisms
    $
        [\Pin_G]\longrightarrow [\Pin_H]
    $
    are equivalence classes $[f]$ of homomorphisms of pinned, quasi-split $F$-groups
    $
        f:\Pin_G\longrightarrow \Pin_H
    $,
    under the equivalence relation described above.

    \item Composition is defined by
    \[
        [g]\circ [f] := [g\circ f].
    \]

    \item The identity morphism on $[\Pin_G]$ is $[\id_{\Pin_G}]$.
\end{itemize}
Let $E_{\qs,\pin}(F)_m$ be the full subcategory of $E_{\qs,\pin}(F)$ whose objects are the classes $[\Pin_G]$ such that the underlying group $G$ splits over an at most $m$-ramified extension of $F$. Let $\mathcal F_m : E_{\qs,\pin}(F)_m \to E_{\qs,\pin}(F')_m$ be the functor defined by $\mathcal F_m([\Pin_G]) = [\Pin_{G'}]$ and $\mathcal F_m([f]) = [f']$.
In summary, we have proved the following:
\begin{theorem}
    The functor $\mathcal F_m : E_{\qs,\pin}(F)_m \to E_{\qs,\pin}(F')_m$ is an equivalence of categories.
\end{theorem}

Let $A$ be a split torus inside the center of $G$. Fix a pinning $\Pin_G = (G,B,T,\{x_{\alpha}\})$ of $G$ such that $A \subseteq T$. Let $\pi : G \to G/A$ be the quotient map. Then $\pi(T) = T/A$ is a maximal torus of $G/A$, and $\pi(B) = B/A$ is a Borel subgroup containing $\pi(T)$. For any $\alpha \in \Phi(G,T)$, the map $x_{\alpha} : \mathbb G_a \to U_{\alpha}$ satisfies $t x_{\alpha}(a) t^{-1} = x_{\alpha}(\alpha(t)a)$, where $t \in T_{F_s}(D)$, $a \in \mathbb G_a(D)$, and $D$ is any $F_s$-algebra. It follows that $A \subseteq \ker(\alpha)$, and $\alpha : T_{F_s} \to \mathbb G_m$ induces a unique character $\bar{\alpha} : (T/A)_{F_s} \to \mathbb G_m$. This induces canonical isomorphisms $\Phi(G,T) \cong \Phi(G/A,T/A)$ and $\Delta(B,T) \cong \Delta(B/A,T/A)$. Define $\bar{x}_{\bar{\alpha}} := \pi \circ x_{\alpha} : \mathbb G_a \to U_{\bar{\alpha}}$. Then $\ker(\bar{x}_{\bar{\alpha}}) = x_{\alpha}^{-1}(A \cap U_{\alpha}) = \{0\}$. As $\pi$ is surjective, we get that $\bar{x}_{\bar{\alpha}} : \mathbb{G}_a \to U_{\bar{\alpha}}$ is an isomorphism. It follows that $\Pin_{G/A} = (G/A, B/A, T/A, \{\bar{x}_{\bar{\alpha}}\})$ is a pinning of $G/A$. We have $\Lie(G_{F_s}) = \Lie(T_{F_s}) \oplus \bigoplus_{\alpha \in \Phi(G,T)} \Lie(U_{\alpha})$. Since $A \subset T$, it follows that $\ker (d \pi_s) = \Lie(A_{F_s}) \subset \Lie(T_{F_s})$. Thus, $\ker(d \pi_s) \cap \Lie(U_{\alpha}) = 0$. Then it follows that $d \pi_s \vert_{\Lie(U_{\alpha})} \ne 0$ for all $\alpha \in \Phi(G,T)$. Hence, $\pi : G \to G/A$ induces a homomorphism $\Pin_G \to \Pin_{G/A}$ of pinned groups. Note that $\pi$ being a quotient map implies that the induced map $\pi^* : X^*(T/A) \to X^*(T)$ given by $\pi^*(\chi) = \chi \circ \pi\vert_{T}$ is injective. We get the corresponding quotient map $\pi': G' \to (G/A)'$ over $F'$. Let $\pi'^* : X^*((T/A)') \to X^*(T')$ be the map induced by $\pi'$, which satisfies $\pi'^* : \chi' \mapsto \chi' \circ \pi'\vert_{T'}$. Then,
\[\ker(\pi'\vert_{T'}) = \bigcap_{\chi \in X^*((T/A)')}\ker(\chi \circ \pi'\vert_{T'}).\] Set $A' := \ker(\pi'\vert_{T'})$.
The kernel $\ker(\pi')$ is central because $\ker(\pi)$ is central. So, it follows that $\ker(\pi') = \ker(\pi'\vert_{T'})$. We have a $\Del_m$-equivariant isomorphism $X^*(A') \cong X^*(T')/X^*((T/A)') \cong X^*(T)/X^*(T/A) \cong X^*(A)$. So, it follows that $\Gamma_{F'}$ acts trivially on $X^*(A')$ and thus, $A'$ is a split torus.
\begin{corollary}\label{cor: Functrlty_of_quotient_maps}
    \[G'/A' \cong (G/A)'.\]
\end{corollary}
As $A'$ is a split torus, we have $(G/A)'(F') \cong G'(F')/A'(F')$. Thus, we do not distinguish between $(G/A)'(F')$ and $G'(F')/A'(F')$. Also, we have $\dim(A) = \dim (T) - \rank X^*(T/A) = \dim(T') - \rank X^*((T/A)') = \dim(A')$. In particular, $A$ is the maximal split torus in the center if and only if $A'$ is as well.

\section{Measures and motives over close local fields}\label{sec:Motives_over_close_fields}

Let $F$ and $F'$ be such that $\Tr_m(F) \cong \Tr_m(F').$ Let $G \in E_{\qs}(F,G_0)$, and let $G' \in E_{\qs}(F',G_0)$ be the corresponding group. Further, consider $(B,T)$ and $(B',T')$ as in Subsection~\ref{subsec:Gan19}. We will show that Deligne's isomorphism preserves Artin conductors and $L$-functions attached to motives.

\subsection{Artin conductors of Galois representations over close local fields}

\begin{lemma} \label{Lemma:Artin_cond_of_gal_rep_over_close_fields}
    Let $V$ and $V'$ be finite-dimensional representations of $\Gamma_F$ and $\Gamma_{F'}$, respectively, such that $I_F^m$ and $I_{F'}^m$ act trivially. If there exists a $\Del_m$-equivariant isomorphism $f \colon V \to V'$, then $a(V) = a(V')$.
\end{lemma}

\begin{proof}
    Let $L$ be the fixed field of $K = \ker(\Gamma_F \to \mathrm{GL}(V))$, and let $\Delta = \mathrm{Gal}(L/F)$. Let $\Delta^r$ denote the upper ramification subgroups. Similarly, define $L'$, $\Delta'$, and $\Delta'^r$ for $V'$. By Equation~\ref{Eq: Artin_conductor_in_upper_filtration}, we have
    \begin{equation}\label{Eq: Artin_cond_in_upper_filt_2}
        a(V) = \int_{-1}^{\infty} \codim V^{\Delta^r} \, dr.
    \end{equation}
    Note that $I_F^m \subseteq K$.

    As a consequence of Proposition~14, Chapter~4 of \cite{SerreLocalFields}, we get $\Delta^r = (\Gamma_F/K)^r \cong \Gamma_F^r K/K \cong \Gamma_F^r / (K \cap \Gamma_F^r)$. If $r \ge 0$, then $\Delta^r \cong I_F^r K/K \cong I_F^r / (I_F^r \cap K)$. If $r \ge m$, then $I_F^r \subseteq K$, which implies that $\Delta^r = 1$. Also, $\Delta^r = \Delta$ if $r \in [-1,0)$. Then Equation~\ref{Eq: Artin_cond_in_upper_filt_2} can be written as
    $$a(V) = \int_{-1}^m \codim V^{\Delta^r} \, dr.$$
    The map $\Del_m$ induces an isomorphism $\Gamma_F^r/I_F^m \cong \Gamma_{F'}^r/I_{F'}^m$ for $r \le m$. Since the group $K$ acts trivially on $V$, it follows that $V^{\Delta^r} = V^{\Gamma_F^r} = V^{\Gamma_F^r/I_F^m}$. Similarly, ${V'}^{\Delta'^r} = {V'}^{\Gamma_{F'}^r/I_{F'}^m}$. Thus, $f \colon V \to V'$ induces an isomorphism of $V^{\Delta^r} \cong {V'}^{\Delta'^r}$ for all $r \le m$. Hence, $a(V) = a(V')$.
\end{proof}

\begin{remark}\label{Remark: Artin_cond_of_Weil_group_over_close_fields}
    Lemma~\ref{Lemma:Artin_cond_of_gal_rep_over_close_fields} holds if we consider the Weil groups $W_F$ and $W_{F'}$ instead of the absolute Galois groups. See Remark~\ref{Remark: Artin_conductor_of_Weil_group}.
\end{remark}

\begin{remark}
    Lemma~\ref{Lemma:Artin_cond_of_gal_rep_over_close_fields} can be proved using \cite[Proposition~3.7.1]{Deligne84} as follows. Recall that the local epsilon factor $\epsilon_0(V,s)$ satisfies $\epsilon_0(V,s) = \omega(V)q^{a(V) (s - \frac{1}{2})}$. Proposition~3.7.1 of \cite{Deligne84} states that $\epsilon_0(V,s) = \epsilon_0(V \circ \Del_m^{-1},s)$, where the additive characters $\psi$ and $\psi'$ on $F$ and $F'$, respectively, satisfy $\psi \sim_m \psi'$, and the Haar measures $\mu$ and $\mu'$ on $F$ and $F'$, respectively, satisfy $\mu(\mathfrak{o}_F) = \mu'(\mathfrak{o}_{F'})$. The existence of a $\Del_m$-equivariant isomorphism $f \colon V \to V'$ implies that $V'$ and $V \circ \Del_m^{-1}$ are isomorphic, and thus $\epsilon_0(V,s) = \epsilon_0(V \circ \Del_m^{-1},s) = \epsilon_0(V',s)$. Therefore, it follows that $a(V) = a(V')$.
\end{remark}

\subsection{Motives over close local fields}

Recall that there exists a $\Del_m$-equivariant isomorphism $\Phi_m \colon X^*(T) \to X^*(T')$ given by $x \mapsto x \circ \phi \circ \phi'^{-1}$, where $\phi$ and $\phi'$ are as in Subsection~\ref{subsec:Gan19}. The absolute Weyl group $W$ corresponding to $(G,T)$ can be realized as a subgroup of $\Aut(X^*(T))$ generated by the reflections $s_{\alpha} \colon x \mapsto x - \langle x, \alpha^{\vee} \rangle \alpha$ for all $\alpha \in \Delta(B,T)$. The $\Del_m$-equivariant isomorphism $\Phi_m \colon X^*(T) \to X^*(T')$ induces an isomorphism from $\Aut(X^*(T))$ to $\Aut(X^*(T'))$ given by $f \mapsto \Phi_m \circ f \circ \Phi_m^{-1}$. Restricting this isomorphism to $W$, we get an isomorphism of the absolute Weyl groups $i_m \colon W \to W'$. More specifically, let $\alpha' = \Phi_m(\alpha)$ for $\alpha \in \Delta(B,T)$. Then for $x \in X^*(T')$,
\begin{align*}
    i_m(s_{\alpha})(x) &= \Phi_m (s_{\alpha}(\Phi_m^{-1}(x))) \\
    &= x - \langle \Phi_m^{-1}(x), \alpha^{\vee} \rangle \Phi_m(\alpha) \\
    &= x - \langle x, \Phi_m(\alpha)^{\vee} \rangle \Phi_m(\alpha) \\
    &= x - \langle x , {\alpha'}^{\vee} \rangle \alpha' \\
    &= s_{\alpha'}(x).
\end{align*}
Thus, we have $i_m(s_{\alpha}) = s_{\alpha'}$. The Galois group $\Gamma_F$ acts on $W$ via the formula
$$\sigma \cdot s_{\alpha} = s_{\sigma (\alpha)}.$$
We claim that the map $i_m \colon W \to W'$ is $\Del_m$-equivariant. To see this, note that for $\sigma \in \Gamma_F/I_F^m$ and $\sigma' = \Del_m(\sigma)$,
\begin{align*}
    i_m(\sigma \cdot s_{\alpha}) &= i_m(s_{\sigma(\alpha)}) \\
    &= s_{\Phi_m(\sigma (\alpha))} \\
    &= s_{\sigma' \Phi_m(\alpha)} \\
    &= s_{\sigma' \alpha'} \\
    &= \sigma' s_{\alpha'} \\
    &= \sigma' \cdot i_m(s_{\alpha}).
\end{align*}
Let $E = X^*(T) \otimes \mathbb{Q}$ and $E' = X^*(T') \otimes \mathbb{Q}$. The $\Del_m$-equivariant isomorphism $X^*(T) \simeq X^*(T')$ induces a $\Del_m$-equivariant isomorphism $E \simeq E'$. This $\Del_m$-equivariant isomorphism induces a canonical isomorphism $f_m^k \colon \Sym^k(E) \to \Sym^k(E')$, where $k$ is a non-negative integer. Under the map $\Phi_m$, $s_{\alpha}(x) \mapsto \Phi_m(x) - \langle x, \alpha^\vee \rangle \Phi_m(\alpha) = \Phi_m(x) - \langle \Phi_m(x), \Phi_m(\alpha)^\vee \rangle \Phi_m(\alpha) = s_{\alpha'}(\Phi_m(x))$. This implies that the map $\Phi_m \colon X^*(T) \to X^*(T')$ is an $i_m$-equivariant map. Consequently, the map $f_m^k \colon \Sym^k(E) \to \Sym^k(E')$ is an $i_m$-equivariant map. Hence, the restricted map $f_m^k \colon \Sym^k(E)^W \to \Sym^k(E')^{W'}$ is well-defined and is a $\Del_m$-equivariant isomorphism. Let $R = \Sym^*(E)^W$ and $R' = \Sym^*(E')^{W'}$. Let $f_m \colon R \to R'$ be the canonical map. Then $f_m$ is a $\Del_m$-equivariant graded ring isomorphism. Let $R_{+}$ be the ideal of elements of positive degree in $R$, and let $V = R_{+}/R_{+}^2 = \bigoplus_{d \ge 1} V_d$. Similarly, let $V' = R'_+/{R'_+}^2 = \bigoplus_{d \ge 1} V'_d$. Define the induced map $\bar{f}_m \colon V \to V'$ by $\bar{f}_m(x + R_{+}^2) = f_m(x) + {R'_+}^2$ for $x \in R_{+}$.

\begin{proposition}
    The restriction of $\bar{f}_m$ to $V_d$ is a well-defined $\Del_m$-equivariant isomorphism $V_d \to V'_d$.
\end{proposition}
\begin{proof}
    Because $f_m \colon R \to R'$ is a graded ring isomorphism, it maps the ideal $R_{+}$ bijectively onto $R'_{+}$. Furthermore, as a ring isomorphism, it preserves products, meaning $f_m(R_{+}^2) = {R'_+}^2$. Therefore, if $x, y \in R_{+}$ are elements such that $x - y \in R_{+}^2$, we have $f_m(x) - f_m(y) = f_m(x - y) \in {R'_+}^2$. This ensures that the induced map $\bar{f}_m \colon V \to V'$ is well-defined and injective. Since $f_m$ is surjective onto $R'_+$, the induced map $\bar{f}_m$ is surjective, making it a vector space isomorphism.

    Let $R_d = \Sym^d(E)^{W}$ and $R'_d = \Sym^d(E')^{W'}$. As $f_m$ is a graded map, it maps the degree $d$ homogeneous component $R_d$ isomorphically onto $R'_d$. Note that $V_d = \{x + R_+^2 : x \in R_d\}$. For $x \in R_d$, the element $\bar{f}_m(x + R_+^2) = f_m(x) + {R'_+}^2$ lies in $V'_d$ because $f_m(x) = f_m^d(x) \in R'_d$. This implies that the restriction $\bar{f}_m \colon V_d \to V'_d$ is well-defined. It follows that $\bar{f}_m \colon V_d \to V'_d$ is $\Del_m$-equivariant since $f_m$ is also $\Del_m$-equivariant.
\end{proof}

\begin{proposition}\label{Artin_conductor_of_motive_are_equal}
    We have $a(M_G) = a(M_{G'}).$
\end{proposition}

\begin{proof}
    By definition, we have $$a(M_G) = \sum_{d \ge 1} (2d-1)a(V_d)$$ and $$a(M_{G'}) = \sum_{d \ge 1} (2d-1)a(V'_d),$$ where $V_d$ and $V'_d$ are $\mathbb{Q}[\Gamma_F]$- and $\mathbb{Q}[\Gamma_{F'}]$-modules, respectively. To show that $a(M_G) = a(M_{G'})$, it is sufficient to prove that $a(V_d) = a(V'_d)$. By Lemma~\ref{Lemma:Artin_cond_of_gal_rep_over_close_fields}, we have $a(V_d) = a(V'_d)$.
\end{proof}
\begin{proposition}\label{Prop:Motivic_L_values_are_equal}
    We have \[ L(M_G^{\vee}(1)) = L(M_{G'}^{\vee}(1)). \]
\end{proposition}
\begin{proof}
    The motive of $\overline{\mathcal{G}}^{\Red}$ over $\mathfrak{f}$ is given by $M_G^{I_F}$~\cite[Proposition~4.5]{Gross97}. By the definition of $L(M_G^{\vee}(1))$ and Steinberg's formula~\cite[Section~3]{Gross97}, we obtain the following: 
    \[ L(M_G^{\vee}(1))^{-1} = \prod_{d \ge 1} \det(1 - \Fr q^{-d} \mid V_d^{I_F}). \]
    A similar formula holds for $L(M_{G'}^{\vee}(1))$, where we consider $\Fr' = \Del_m(\Fr)$. 
    It is sufficient to show that 
    \[ \det(1 - \Fr q^{-d} \mid V_d^{I_F}) = \det(1 - \Fr' q^{-d} \mid {V'_d}^{I_{F'}}). \]
    Note that $V_d^{I_F} = V_d^{I_F/I_F^m}$, and a similar expression holds for ${V'_d}^{I_{F'}}$. Thus, the map $\bar{f}_m \colon V_d \to V'_d$ induces a $\Del_m$-equivariant isomorphism $V_d^{I_F} \to {V'_d}^{I_{F'}}$. This yields the commutative diagram
\[\begin{tikzcd}[ampersand replacement=\&]
	{V_d^{I_F}} \& {{V'_d}^{I_{F'}}} \\
	{V_d^{I_F}} \& {{V'_d}^{I_{F'}}}
	\arrow["\bar{f}_m", from=1-1, to=1-2]
	\arrow["{1 - q^{-d}\Fr}"', from=1-1, to=2-1]
	\arrow["{1 - q^{-d}\Fr'}", from=1-2, to=2-2]
	\arrow["\bar{f}_m", from=2-1, to=2-2]
\end{tikzcd}\]
    It follows that $\det(1 - \Fr q^{-d} \mid V_d^{I_F}) = \det(1 - \Fr' q^{-d} \mid {V'_d}^{I_{F'}})$.
\end{proof}

\subsection{Gross's measure over close local fields}
Recall that the measure $\lvert \omega_G \rvert$ depends on a specific special vertex chosen as described in Section~4 of \cite{Gross97}. The simplicial isomorphism of apartments constructed in Proposition~4.4 of \cite{Ganapathy19} preserves this choice of special vertex. In particular, we may use Theorem~4.5 of \cite{Ganapathy19}. Specifically, let $\mathfrak{v}$ be the special vertex in $\mathcal{A}^{\Red}(S,F)$ as recalled earlier. Let $P = \mathcal{G}_{\mathfrak{v}}(\mathfrak{o}_F)$ be the parahoric subgroup.

\begin{theorem}\label{Thm:special_vertex_choice_and_close_parahoric}
    Let $m \ge 1$. There exists $e \ge m$ such that if $F \sim_e F'$, then the following holds:
    \begin{enumerate}
        \item The vertex $\mathfrak{v}' := \mathcal{A}_m(\mathfrak{v}) \in \mathcal{A}^{\Red}(S',F')$ is the special vertex chosen for $G'$ as in \cite[Section~4]{Gross97}.
        \item The corresponding parahoric group schemes $\mathcal{G}_{\mathfrak{v}} \times_{\mathfrak{o}_F} \mathfrak{f}_m$ and $\mathcal{G}'_{\mathfrak{v}'} \times_{\mathfrak{o}_{F'}} \mathfrak{f}'_m$ are isomorphic. In particular, $\mathcal{G}_{\mathfrak{v}}(\mathfrak{f}_m) \cong \mathcal{G}'_{\mathfrak{v}'}(\mathfrak{f}'_m)$ as groups.
    \end{enumerate}
\end{theorem}

\begin{proof}
    The first part follows from the description of the special vertex chosen in Section~4 of \cite{Gross97} and the description of the isomorphism $\mathcal{A}_m \colon \mathcal{A}^{\Red}(S,F) \to \mathcal{A}^{\Red}(S',F')$ in Proposition~4.4 of \cite{Ganapathy19}. The second part follows from Theorem~4.5 of \cite{Ganapathy19}.
\end{proof}

\begin{corollary} \label{Measures_are_equal_close_field_type1}
    If $\mu_{\psi}(\mathfrak{o}_F) = \mu_{\psi'}(\mathfrak{o}_{F'})$, then
    \[ \int_{P_m} \lvert \omega_G \rvert = \int_{P'_m} \lvert \omega_{G'} \rvert. \]
\end{corollary}

\begin{proof}
    This follows from Corollary~\ref{Volume_of_congruence_subgroup}, Proposition~\ref{Prop:Motivic_L_values_are_equal}, and Theorem~\ref{Thm:special_vertex_choice_and_close_parahoric}.
\end{proof}

\section{Formal degrees over close local fields}\label{sec:Formal_degrees_over_close_fields}

Set $m \ge 1$. Let $\sigma$ be an irreducible admissible representation of $G(F)$ such that $\sigma^{P_m} \ne 0$. By Theorem~\ref{KazIsomorphism}, there exists an integer $l \ge m$ such that if $F$ and $F'$ are $l$-close, then we have a Hecke algebra isomorphism $\Kaz_m \colon \mathcal{H}(G, P_m) \to \mathcal{H}(G', P'_m)$. Thus, this induces a bijection:
\[
\begin{gathered}
\left\{
\begin{array}{c}
\text{Irreducible admissible representations } (\sigma,V) \text{ of } G(F) \\
\text{such that } \sigma^{P_m} \ne 0
\end{array}
\right\} \\
\updownarrow \\
\left\{
\begin{array}{c}
\text{Irreducible admissible representations } (\sigma',V') \text{ of } G'(F') \\
\text{such that } {\sigma'}^{P'_m} \ne 0
\end{array}
\right\}.
\end{gathered}
\]

We fix Haar measures $dg$ and $dg'$ on $G(F)$ and $G'(F')$, respectively, such that $\operatorname{Vol}(P_m, dg) = \operatorname{Vol}(P'_m, dg') = 1$. Similarly, we fix Haar measures $dz$ and $dz'$ on $Z(F)$ and $Z'(F')$, respectively, such that $\operatorname{Vol}(Z(F) \cap P_m, dz) = \operatorname{Vol}(Z'(F') \cap P'_m, dz') = 1$. We also define measures $d\bar{g}$ and $d\bar{g}'$ on the quotients $Z(F)\backslash G(F)$ and $Z'(F') \backslash G'(F')$ such that $dg = dz \, d\bar{g}$ and $dg' = dz' \, d\bar{g}'$, respectively. Let $\Pi_*^2(G,F)_m$ denote the set of equivalence classes of discrete series representations of $G(F)$ that have non-zero $P_m$-fixed vectors. We prove the following:
\begin{theorem}\label{formal_degrees_are_equal1}
    Let $\sigma$ and $\sigma'$ correspond via the bijection mentioned above. Then $\sigma \in \Pi_*^2(G,F)_m$ if and only if $\sigma' \in \Pi_*^2(G', F')_m$. In this case,
    \[
    d(\sigma, d\bar{g}) = d(\sigma', d\bar{g}').
    \]
\end{theorem}

\begin{proof}
This generalizes Theorem~4.6 of~\cite{Ganapathy15}, which is stated for split reductive groups. Let $v \in V^{P_m}$ and $v^{\vee} \in (V^{\vee})^{P_m}$ be such that $\langle v, v^{\vee} \rangle \ne 0$. Let $h_{\sigma}(g) = \langle \sigma(g)v, v^{\vee} \rangle$. Note that $h_{\sigma}$ is constant on the double coset $P_m g P_m$, and therefore $h_{\sigma}(g) = \operatorname{Vol}(P_m g P_m; dg)^{-1} \langle \sigma(t_g)v, v^{\vee} \rangle$, where $t_g$ is the characteristic function on $P_m g P_m$.

By Theorem~\ref{thm:Cartan_decomp}, we have $G(F) = \bigsqcup_{\tau \in \Omega_{T,+}} G_{\tau}$. Write 
\[ 
G(F) = \bigsqcup_{\tau \in \Omega_{T,+}}\bigsqcup_{(k_i,k_j) \in T_{\tau}} P_m k_i n_{\tau} k_j^{-1}P_m. 
\]
Note that $X_*(Z)$ canonically sits inside $\bar{\Omega}_{T,+}$. Every element $z \in Z(F)$ can be uniquely written as $z = z_1 p(\mu)$, where $\mu \in X_*(Z)$ and $z_1 \in Z(\mathfrak{o}_F)$. Let $\tau = \tau_1 \oplus \tau_2 \in \bar{\Omega}_{T,+} \oplus (\Omega_T)_{\tor}$. Then 
\[ 
z P_m k_i n_{\tau} k_j^{-1}P_m = z_1 P_m k_i p(\mu) p(\tau_1) s(\tau_2)k_j^{-1}P_m. 
\]
Since $p$ is a group homomorphism, it follows that 
\[ 
z P_m k_i n_{\tau} k_j^{-1}P_m = z_1 P_m k_i p(\mu + \tau_1) s(\tau_2)k_j^{-1}P_m = z_1 P_m k_i n_{\tau+ \mu} k_j^{-1} P_m. 
\]
Thus, 
\[
z P_m k_i n_{\tau} k_j^{-1}P_m =
\begin{cases}
P_m k_i n_{\mu +\tau } k_j^{-1}P_m, & \text{if } z_1 \in Z(F) \cap P_m, \\
P_m l_i n_{\mu +\tau } l_j^{-1}P_m, & \text{if } z_1 \notin Z(F) \cap P_m,
\end{cases}
\]
where $(l_i,l_j)$ is an element of $T_{\mu + \tau }$ that is uniquely determined by the class of $z_1 \bmod Z(F) \cap P_m$.

Let $A_F$ denote the set of representatives of the elements of $\Omega_{T,+}/X_*(Z)$. Then we have 
\begin{align*}
    \int_{Z(F) \backslash G(F)} \lvert h_{\sigma}(g) \rvert^2 \, d\bar{g} 
    &= \sum_{\tau \in A_F} \sum_{(k_i,k_j) \in T_{\tau}}  
       \frac{\operatorname{Vol}(P_m k_i n_{\tau} k_j^{-1}P_m; dg)}
            {\# Z(\mathfrak{f}_m) \operatorname{Vol}(Z(F) \cap P_m; dz)} 
       \lvert h_{\sigma}(k_i n_{\tau} k_j^{-1}) \rvert^2 \\
    &= \sum_{\tau \in A_F} \sum_{(k_i,k_j) \in T_{\tau}}  
       \frac{\operatorname{Vol}(P_m k_i n_{\tau} k_j^{-1}P_m; dg)}
            {\# Z(\mathfrak{f}_m)} 
       \lvert h_{\sigma}(k_i n_{\tau} k_j^{-1}) \rvert^2 \\
    &= \sum_{\tau \in A_F} \sum_{(k_i,k_j) \in T_{\tau}}  
       \frac{\operatorname{Vol}(P_m k_i n_{\tau} k_j^{-1}P_m; dg)^{-1}}
            {\# Z(\mathfrak{f}_m)} 
       \lvert \langle \sigma(t_{k_i n_{\tau} k_j^{-1}})v, v^{\vee} \rangle \rvert^2.
\end{align*}

Let $\gamma \colon \sigma^{P_m} \to {\sigma'}^{P'_m}$ be the isomorphism induced by $\Kaz_m$. This gives rise to an isomorphism of dual spaces $\gamma^{\vee} \colon (\sigma^{\vee})^{P_m} \to ({\sigma'}^{\vee})^{P'_m}$. We define $h'_{\sigma'}(g') = {v'}^{\vee}(\sigma'(g')v')$. Let $t'_{g'}$ denote the characteristic function of $P'_m g' P'_m$. By Proposition~\ref{Double_coset_Vol_mathcing_prop}, it follows that 
\[
\Vol(P_m k_i n_{\tau} k_j^{-1}P_m; dg) = \Vol(P'_m k'_i n_{\tau'} {k'_j}^{-1}P'_m; dg'),
\]
where $k'_i$ satisfies the equation $k'_i \bmod P'_m = p_m(k_i \bmod P_m)$. It also follows that $\langle \sigma(t_{k_i n_{\tau} k_j^{-1}})v, v^{\vee} \rangle = \langle \sigma'(t'_{k'_i n'_{\tau'} {k'_j}^{-1}})v', {v'}^{\vee} \rangle$.

The isomorphism $\Omega_T \simeq \Omega_{T'}$ induces an isomorphism $\Omega_{T,+} \simeq \Omega_{T',+}$~\cite[p.~329]{Ganapathy22}. Recall that we have an isomorphism $\Phi_m \colon \Phi(G,S) \simeq \Phi(G',S')$. As $X_*(Z) = \{ \lambda \in X_*(S) : \langle \alpha, \lambda \rangle = 0 \quad \forall \alpha \in \Phi(G,S) \}$, it follows that the image of $X_*(Z)$ is $X_*(Z')$ under the above isomorphism. Thus, we have $\Omega_{T,+}/X_*(Z) \cong \Omega_{T',+}/X_*(Z')$. We have $X_*(Z) \cong X_*(Z')$ and $\dim(Z) = \dim X_*(Z) = \dim X_*(Z') = \dim(Z')$. Note that if $F$ and $F'$ are $l$-close as in Theorem~\ref{KazIsomorphism}, then $F$ and $F'$ are $m$-close, and thus $\#Z(\mathfrak{f}_m) = \#Z'(\mathfrak{f}'_m)$. Combining all of the above, we get that 
\[
\int_{Z(F) \backslash G(F)} \lvert h_{\sigma}(g) \rvert^2 \, d\bar{g} = \int_{Z'(F') \backslash G'(F')} \lvert h'_{\sigma'}(g') \rvert^2 \, d\bar{g}'.
\]
Hence, $\sigma \in \Pi^2_*(G,F)$ if and only if $\sigma' \in \Pi^2_*(G',F')$. Now, the proof that $d(\sigma, d\bar{g}) = d(\sigma', d\bar{g}')$ proceeds identically to the proof of Corollary~4.7 of~\cite{Ganapathy15}.
\end{proof}

\section{Parahoric subgroups of $G/Z$}

In this section, we set $\mathfrak{o} := \mathfrak{o}_F$ and $\mathfrak{p} := \mathfrak{p}_F$. Let $\mathfrak{O} := \mathfrak{O}_{\kur}$ be the ring of integers of $\kur$. Let $G$ be a quasi-split, connected reductive $F$-group. Let $Z$ be any split torus in the center of $G$, and let $S$ be an $F$-maximal split torus in $G$.

Let $\Bred(G)$ be the reduced building of $G$. Note that any surjective homomorphism $G \to H$ with central kernel induces an equivariant isomorphism $f \colon \Bred(G) \to \Bred(H)$. Here, $H = G/Z$ satisfies this hypothesis. Let $v \in \Bred(G)$ be a vertex. Then $f(v) \in \Bred(G/Z)$ is a vertex. Recall that the group $G(F)^0$ is the kernel of the Kottwitz homomorphism $\kappa_G$~\cite[Proposition~11.5.4]{KalethaPrasad23}. In particular, the group $G(F)^0$ is functorial with respect to all homomorphisms of reductive groups~\cite[Remark~11.5.5]{KalethaPrasad23}. Let $\pi \colon G \to G/Z$ be the quotient map. Using the fact that $Z$ is a split torus and that the Kottwitz homomorphism $\kappa_G$ is functorial in $G$, we get the following commutative diagram:
\[\begin{tikzcd}[ampersand replacement=\&]
	\& {Z(F)^0} \& {G(F)^0} \& {(G/Z)(F)^0} \& \\
	1 \& {Z(F)} \& {G(F)} \& {(G/Z)(F)} \& 1 \\
	0 \& {(X_*(Z)_{I_F})^{\Fr}} \& {\pi_1(G)_{I_F}^{\Fr}} \& {\pi_1(G/Z)_{I_F}^{\Fr}} \& 0 \\
	\& 1 \& 1 \& 1
	\arrow[from=1-2, to=2-2]
	\arrow[from=1-3, to=2-3]
	\arrow[from=1-4, to=2-4]
	\arrow[from=2-1, to=2-2]
	\arrow[from=2-2, to=2-3]
	\arrow["{\kappa_Z}"', from=2-2, to=3-2]
	\arrow[from=2-3, to=2-4]
	\arrow["{\kappa_G}"', from=2-3, to=3-3]
	\arrow[from=2-4, to=2-5]
	\arrow["{\kappa_{G/Z}}"', from=2-4, to=3-4]
	\arrow[from=3-1, to=3-2]
	\arrow[from=3-2, to=3-3]
	\arrow[from=3-2, to=4-2]
	\arrow[from=3-3, to=3-4]
	\arrow[from=3-3, to=4-3]
	\arrow[from=3-4, to=3-5]
	\arrow[from=3-4, to=4-4]
\end{tikzcd}\]
By the snake lemma, it follows that $\pi(G(F)^0) = (G/Z)(F)^0$. Recall that the parahoric subgroup attached to a vertex $v \in \Bred$ is the stabilizer of $v$ in $G(F)^0$. Let $P_v$ and $Q_v := P_{f(v)}$ be the parahoric subgroups attached to the vertices $v$ and $f(v)$, respectively.

Let $g \in P_v$. Then $\pi(g) \cdot f(v) = f(g \cdot v) = f(v)$. Thus, $\pi(P_v) \subseteq Q_v$. Conversely, let $h \in Q_v$. Then there exists $g \in G(F)^0$ such that $\pi(g) = h$. It follows that $f(g \cdot v) = f(v)$ and thus $g \cdot v = v$. Therefore, $g \in P_v$, and so $Q_v \subseteq \pi(P_v)$. This implies the following:
\begin{claim}\label{claim1}
    $\pi(P_v) = Q_v$.
\end{claim}

Let $\SchP$ and $\SchQ$ be parahoric group $\mathfrak{o}$-schemes such that
\[
\SchP(\mathfrak{o}) = P_v \quad \text{and} \quad \SchQ(\mathfrak{o}) = Q_v
\]
with generic fibers $G$ and $G/Z$, respectively~\cite[Axiom~4.1.20]{KalethaPrasad23}. 
Note that $\ker(\SchP(\mathfrak{o}) \to \SchQ(\mathfrak{o})) = Z(F) \cap \SchP(\mathfrak{o})$. The group $Z(F) \cap \SchP(\mathfrak{o})$ is a bounded subgroup of $Z(F)$, so it is contained in the maximal bounded subgroup $Z(\mathfrak{o})$ of $Z(F)$. Since the group $Z(\mathfrak{o})$ acts trivially on the building, we have $Z(\mathfrak{o}) \subseteq Z(F) \cap \SchP(\mathfrak{o})$. Thus, we have 
\[
\ker(\SchP(\mathfrak{o}) \to \SchQ(\mathfrak{o})) = Z(\mathfrak{o}).
\]
We consider $\mathcal{Z}$, the standard model for $Z$ (see~\cite[Appendix~B-2]{KalethaPrasad23}), as a subgroup of $\mathcal{P}$. 

Let $\pi_{\unr} \colon G_{\kur} \to (G/Z)_{\kur}$ be the natural quotient map after base change. Then $\Bred(G)$ sits inside $\Bred(G_{\kur})$, and the parahoric subgroups are compatible with base change~\cite[Axiom~4.1.27]{KalethaPrasad23}. Let $\mathcal{P}_{\unr} = \mathcal{P} \times_{\mathfrak{o}} \mathfrak{O}$ and $\mathcal{Q}_{\unr} = \mathcal{Q} \times_{\mathfrak{o}} \mathfrak{O}$ be the parahoric group schemes attached to $v$ with generic fibers $G_{\kur}$ and $(G/Z)_{\kur}$, respectively. Similar to Claim~\ref{claim1}, we have $\pi_{\unr}(\mathcal{P}_{\unr}(\mathfrak{O})) = \mathcal{Q}_{\unr}(\mathfrak{O})$. By \cite[Corollary~2.10.10]{KalethaPrasad23}, there exists a unique extension $\Pi \colon \mathcal{P} \to \mathcal{Q}$ of $\pi \colon G \to G/Z$. Similar to above, we have $\ker(\SchP(\mathfrak{O}) \to \SchQ(\mathfrak{O})) = \mathcal{Z}(\mathfrak{O})$ and $\SchP(\mathfrak{O})/\mathcal{Z}(\mathfrak{O}) \cong \SchQ(\mathfrak{O})$. Note that $\mathcal{Z}(\mathfrak{O}) = (\ker \Pi)(\mathfrak{O})$ and $\ker(\Pi) \times_{\mathfrak{o}} F = \ker(\Pi_F) = \mathcal{Z} \times_{\mathfrak{o}} F$. By~\cite[Lemma~2.10.13]{KalethaPrasad23}, we have $\ker(\Pi) = \mathcal{Z}$. Thus, we have a unique map $\bar{\Pi} \colon \mathcal{P}/\mathcal{Z} \to \mathcal{Q}$. Note that taking the generic fiber commutes with base change~\cite[Proposition~9.2~(v)]{SGA3_Tome1}, so $(\mathcal{P}/\mathcal{Z}) \times_{\mathfrak{o}} F \cong (\mathcal{P} \times_{\mathfrak{o}} F)/(\mathcal{Z} \times_{\mathfrak{o}} F) \cong G/Z$. The morphism $\bar{\Pi}$ induces a map 
$\bar{\Pi}(\mathfrak{O}) \colon (\mathcal{P}/\mathcal{Z})(\mathfrak{O}) \to \mathcal{Q}(\mathfrak{O})$. 
We have the following exact sequence $1 \to \mathcal{Z} \to \mathcal{P} \xrightarrow{\phi} \mathcal{P}/\mathcal{Z} \to 1$. Taking $\mathfrak{O}$-points is a left exact functor, so we get the following exact sequence:
\[
1 \to \mathcal{Z}(\mathfrak{O}) \to \mathcal{P}(\mathfrak{O}) \to (\mathcal{P}/\mathcal{Z})(\mathfrak{O}) \to H^1_{\mathrm{fppf}}(\Spec \mathfrak{O}, \mathcal{Z}) \to \cdots
\]
By \cite[\href{https://stacks.math.columbia.edu/tag/03P7}{Section~03P7}]{stacks-project}, we have $H^1_{\mathrm{fppf}}(\Spec \mathfrak{O}, \mathbb{G}_m) = \Pic(\mathfrak{O})$, and thus, $H^1_{\mathrm{fppf}}(\Spec \mathfrak{O}, \mathcal{Z}) \cong \Pic(\mathfrak{O})^{\dim \mathcal{Z}}$. For PIDs, the Picard group is trivial. Therefore, 
\[
1 \to \mathcal{Z}(\mathfrak{O}) \to \mathcal{P}(\mathfrak{O}) \xrightarrow{\phi(\mathfrak{O})} (\mathcal{P}/\mathcal{Z})(\mathfrak{O}) \to 1
\] is exact.
We have the following commutative diagram:
\[\begin{tikzcd}[ampersand replacement=\&]
	{\mathcal P(\mathfrak O)} \& {(\mathcal P/\mathcal Z)(\mathfrak O)} \\
	{\mathcal Q(\mathfrak O)}
	\arrow["{\phi(\mathfrak O)}", from=1-1, to=1-2]
	\arrow["{\Pi(\mathfrak O)}"', from=1-1, to=2-1]
	\arrow["{\bar \Pi(\mathfrak O)}", from=1-2, to=2-1]
\end{tikzcd}\] The surjectivity of $\Pi(\mathfrak{O})$ implies that $\bar{\Pi}(\mathfrak{O})$ is surjective. The surjectivity of $\phi(\mathfrak{O})$ yields the following commutative diagram:
\[\begin{tikzcd}[ampersand replacement=\&]
	{\mathcal P(\mathfrak O)/\mathcal Z(\mathfrak O)} \& {} \& {(\mathcal P/\mathcal Z)(\mathfrak O)} \\
	{\mathcal Q(\mathfrak O)}
	\arrow["{x\mathcal Z(\mathfrak O) \mapsto \phi(\mathfrak O)(x)}", from=1-1, to=1-3]
	\arrow["{x\mathcal Z(\mathfrak O) \mapsto \Pi(\mathfrak O)(x)}"', from=1-1, to=2-1]
	\arrow["{\bar \Pi(\mathfrak O)}", from=1-3, to=2-1]
\end{tikzcd}\] where the vertical and horizontal arrows are isomorphisms. This proves the injectivity of $\bar{\Pi}(\mathfrak{O})$. Since the map $\Pi \colon \mathcal{P} \to \mathcal{Q}$ is the extension of $G \to G/Z$, it follows that the map $\bar{\Pi} \colon \mathcal{P}/\mathcal{Z} \to \mathcal{Q}$ is an extension of the identity map $G/Z \to G/Z$. Therefore, we view both $(\mathcal{P}/\mathcal{Z})(\mathfrak{O})$ and $\mathcal{Q}(\mathfrak{O})$ as subsets of $(G/Z)(\kur)$. Thus, $(\mathcal{P}/\mathcal{Z})(\mathfrak{O}) = \mathcal{Q}(\mathfrak{O})$ as subsets of $(G/Z)(\kur)$.
By Corollary~2.10.11 of~\cite{KalethaPrasad23}, it follows that $\mathcal{P}/\mathcal{Z} \cong \mathcal{Q}$.

Let $Z_m = \ker(\mathcal{Z}(\mathfrak{o}) \to \mathcal{Z}(\mathfrak{o}/\mathfrak{p}^m))$, $P_m = \ker(\mathcal{P}(\mathfrak{o}) \to \mathcal{P}(\mathfrak{o}/\mathfrak{p}^m))$, and $Q_m = \ker(\mathcal{Q}(\mathfrak{o}) \to \mathcal{Q}(\mathfrak{o}/\mathfrak{p}^m))$. We will show that $P_m Z(F)/Z(F)$ is isomorphic to $Q_m$.
The map $\Pi \colon \SchP \to \SchQ$ induces a commutative diagram:

\begin{equation}\label{diag1}
\begin{tikzcd}
\SchP(\mathfrak{o}) \arrow[r, "\pi"] \arrow[d] &
\SchQ(\mathfrak{o}) \arrow[d] \\
\SchP(\mathfrak{o}/\mathfrak{p}^m) \arrow[r] &
\SchQ(\mathfrak{o}/\mathfrak{p}^m)
\end{tikzcd}
\end{equation}
The surjectivity of $\mathcal{P}(\mathfrak{o}) \to \mathcal{Q}(\mathfrak{o})$ and the surjectivity of the vertical arrows in Diagram~\ref{diag1} imply the surjectivity of $\mathcal{P}(\mathfrak{o}/\mathfrak{p}^m) \to \mathcal{Q}(\mathfrak{o}/\mathfrak{p}^m)$. This yields the following commutative diagram:
\begin{center}
\begin{tikzcd}
1 \arrow[r] & \mathcal{Z}(\mathfrak{o}) \arrow[r] \arrow[d] & \mathcal{P}(\mathfrak{o}) \arrow[r] \arrow[d] & \mathcal{Q}(\mathfrak{o}) \arrow[r] \arrow[d] & 1 \\
1 \arrow[r] & \mathcal{Z}(\mathfrak{o}/\mathfrak{p}^m) \arrow[r] \arrow[d] & \mathcal{P}(\mathfrak{o}/\mathfrak{p}^m) \arrow[r] \arrow[d] & \mathcal{Q}(\mathfrak{o}/\mathfrak{p}^m) \arrow[r] \arrow[d] & 1 \\
& 1 & 1 & 1 &
\end{tikzcd}
\end{center}
By the snake lemma, it follows that the sequence $1 \to Z_m \to P_m \to Q_m \to 1$ is exact. Therefore, we have 
\[
Q_m \cong P_m/Z_m = P_m/(P_m \cap Z(F)) \cong P_m Z(F)/Z(F).
\]
Thus, we have proved the following theorem:
\begin{theorem}\label{Thm:Parahoric_subgroups_of_quotients}
    Let $\SchP$, $\SchQ$, and $\mathcal{Z}$ be as above. Then we have:
    \begin{enumerate}
        \item The parahoric subgroups of $G/Z$ attached to vertices are of the form $(\SchP/\mathcal{Z})(\mathfrak{o}) \cong \SchP(\mathfrak{o})/\mathcal{Z}(\mathfrak{o})$.
        \item $\SchP/\mathcal{Z} \cong \SchQ$.
        \item $Q_m \cong P_m Z(F)/Z(F)$.
    \end{enumerate}
\end{theorem}

\section{Parameters over close local fields}\label{sec:Parameters_over_close_fields}
Let $F \sim_m F'$. Let $G \in E_{\qs}(G_0, F)_m$, and let $G' \in E_{\qs}(G_0, F')_m$ be its image under the isomorphism mentioned in Section~\ref{subsec:Gan19}. Fix a pinning $\mathcal{E}$ of $\hat{G}$ and a corresponding pinning $\mathcal{E}'$ of $\hat{G}'$.

We have a $\Del_m$-equivariant isomorphism $\mathcal{R}(\hat{G}) \to \mathcal{R}(\hat{G}')$ of based root data. This induces a unique $\Del_m$-equivariant isomorphism $\psi_m \colon \hat{G} \to \hat{G}'$ preserving pinnings. This induces an isomorphism $\Aut(\hat{G}, \mathcal{E}) \to \Aut(\hat{G}', \mathcal{E}')$ given by $f \mapsto \psi_m \circ f \circ \psi_m^{-1}$. Let $\rho_G \colon \Gamma_F \to \Aut(\hat{G}, \mathcal{E})$ and $\rho_{G'} \colon \Gamma_{F'} \to \Aut(\hat{G}', \mathcal{E}')$ be the corresponding actions defining ${}^L G$ and ${}^L G'$, respectively. Both actions factor through $I_F^m$ and $I_{F'}^m$, respectively. The $\Del_m$-equivariance of the map $\psi_m$ implies that the following diagram commutes: 
\[\begin{tikzcd}[ampersand replacement=\&]
	{\Gamma_F/I_F^m} \& {\Aut(\hat G,\mathcal E)} \\
	{\Gamma_{F'}/I_{F'}^m} \& {\Aut(\hat G',\mathcal E')}
	\arrow["{\rho_G}", from=1-1, to=1-2]
	\arrow["{\Del_m}"', from=1-1, to=2-1]
	\arrow["{f \mapsto \psi_m \circ f \circ \psi_m^{-1}}", from=1-2, to=2-2]
	\arrow["{\rho_{G'}}"', from=2-1, to=2-2]
\end{tikzcd}\]
Let ${}^L G_m := \Gamma_F/I_F^m \ltimes \hat{G}$. It follows that the map ${}^L G_m \to {}^L G'_m$ given by $(\sigma,g) \mapsto (\Del_m(\sigma), \psi_m(g))$ is an isomorphism. We define a parameter $\varphi' \colon W_{F'}/I_{F'}^m \times \SL_2 \to {}^L G'$ such that the following diagram commutes:
\[\begin{tikzcd}[ampersand replacement=\&]
	{W_F/I_F^m \times \SL_2} \& {^LG_m} \\
	{W_{F'}/I_{F'}^m \times \SL_2} \& {^LG'_m}
	\arrow["\varphi", from=1-1, to=1-2]
	\arrow["{(\Del_m, id)}"', from=1-1, to=2-1]
	\arrow["{(\Del_m, \psi_m)}", from=1-2, to=2-2]
	\arrow["{\varphi'}"', from=2-1, to=2-2]
\end{tikzcd}\]
It follows that the map $\mathcal{L}_m(G,F) \to \mathcal{L}_m(G',F')$ given by $\varphi \mapsto \varphi'$ is a bijection. If $\varphi$ and $\varphi'$ are related in this manner, we write $\varphi \sim_m \varphi'$.

If $\varphi \in \mathcal{L}_m(G,F)$, then note that $S_{\varphi} = \{g \in \hat{G} : g \varphi(x) = \varphi(x)g \text{ for all } x \in (W_F/I_F^m) \times \SL_2 \}$.
Whenever $\varphi \sim_m \varphi'$, the canonical isomorphism $\psi_m \colon \hat{G} \to \hat{G}'$ induces an isomorphism $S_{\varphi} \cong S_{\varphi'}$ that takes $Z(\hat{G})^{\Gamma_F/I_F^m}$ to $Z(\hat{G}')^{\Gamma_{F'}/I_{F'}^m}$. Thus, we have
\[
\overline{S_{\varphi}} \cong \overline{S_{\varphi'}}.
\]
Similarly, we have $S_{\varphi}^{\sharp} \cong S_{\varphi'}^{\sharp}$ and
\[
\mathcal{S}_{\varphi}^{\sharp} \cong \mathcal{S}_{\varphi'}^{\sharp}.
\]
As a consequence, we obtain the following proposition:

\begin{proposition}\label{param_over_close_fields}
    The parameter $\varphi$ is discrete if and only if $\varphi'$ is discrete. Furthermore, $\mathcal{S}_{\varphi}^{\sharp} \cong \mathcal{S}_{\varphi'}^{\sharp}$ and $\overline{S_{\varphi}} \cong \overline{S_{\varphi'}}$.
\end{proposition}

Let $V = \Lie(\hat{G})/\Lie(Z(\hat{G})^{\Gamma_F})$. Define $V'$ for $\hat{G}'$ similarly. Let $\Ad \colon {}^L G \to \GL(V)$ be the adjoint representation defined in the following way. Let $d\rho_G(\sigma) \colon \Lie(\hat{G}) \to \Lie(\hat{G})$ denote the derivative of $\rho_G(\sigma) \colon \hat{G} \to \hat{G}$ at the identity. For simplicity of notation, we denote $d\rho_G(\sigma)$ by $d\sigma$. 
Let $\Ad_{\hat{G}} \colon \hat{G} \to \GL(\Lie(\hat{G}))$ be the adjoint representation of $\hat{G}$. Then $\Ad \colon {}^L G \to \GL(V)$ is given by the formula
\[
\Ad(\sigma,g)([X]) = [\Ad_{\hat{G}}(g) (d\sigma(X))],
\]
for $(\sigma,g) \in {}^L G$ and $X \in \Lie(\hat{G})$, where $[X] = X + \Lie(Z(\hat{G})^{\Gamma_F})$.
Similarly, define $\Ad' \colon {}^L G' \to \GL(V')$. Note that $\Ad$ is self-dual and factors through ${}^L G_m$. The $\Del_m$-equivariant isomorphism $\psi_m \colon \hat{G} \simeq \hat{G}'$ induces a $\Del_m$-equivariant isomorphism $d\psi_m \colon \Lie(\hat{G}) \simeq \Lie(\hat{G}')$ that takes $\Lie(Z(\hat{G})^{\Gamma_F})$ to $\Lie(Z(\hat{G}')^{\Gamma_{F'}})$. Thus, we have a $\Del_m$-equivariant isomorphism $\bar{d\psi_m} \colon V \to V'$. We claim that the following diagram commutes:
\begin{equation}\label{Fig: Ad_commutes_with_LG_m}
\begin{tikzcd}[ampersand replacement=\&]
    {}^L G_m \& \GL(V) \\
    {}^L G'_m \& \GL(V')
    \arrow["\Ad", from=1-1, to=1-2]
    \arrow["\simeq"', from=1-1, to=2-1]
    \arrow["\simeq", from=1-2, to=2-2]
    \arrow["\Ad'"', from=2-1, to=2-2]
\end{tikzcd}
\end{equation}
To see this, first note that the following diagram commutes: 
\begin{equation}\label{Fig:sigma_commutes_with_psi_m}
\begin{tikzcd}[ampersand replacement=\&]
    \hat{G} \& \hat{G} \\
    \hat{G}' \& \hat{G}'
    \arrow["\rho_G(\sigma)", from=1-1, to=1-2]
    \arrow["\psi_m"', from=1-1, to=2-1]
    \arrow["\psi_m", from=1-2, to=2-2]
    \arrow["\rho_{G'}(\sigma')"', from=2-1, to=2-2]
\end{tikzcd}
\end{equation}
where $\sigma' = \Del_m(\sigma)$. Taking the derivative at the identity, we obtain
\begin{equation}\label{Eq:dsigma_commutes_with_dpsi}
    d\psi_m \circ d\sigma = d\sigma' \circ d\psi_m.
\end{equation}
Similarly, we obtain 
\begin{equation}\label{Eq:Ad_commutes_with_dpsi}
    d\psi_m \circ \Ad_{\hat{G}}(g) = \Ad_{\hat{G}'}(g') \circ d\psi_m,
\end{equation}
where $g' = \psi_m(g)$ and $g \in \hat{G}$. Thus, we get the following:
\[
d\psi_m \circ \Ad_{\hat{G}}(g) \circ d\sigma = \Ad_{\hat{G}'}(g') \circ d\sigma' \circ d\psi_m.
\]
Hence, the claim follows.

Recall that the representation $\Ad \varphi \colon \WD_F \xrightarrow{\varphi} {}^L G \to \GL(V)$ is a self-dual representation. Note that $\Ad \varphi$ factors through $W_F/I_F^m \times \SL_2$ whenever $\varphi \in \mathcal{L}_m(G,F)$. 
Whenever $\varphi \sim_m \varphi'$, we claim that the following diagram commutes:
\begin{equation}\label{fig:Adphi_commutes}
\begin{tikzcd}[ampersand replacement=\&]
	{W_F/I_F^m \times \SL_2} \& {\GL(V)} \\
	{W_{F'}/I_{F'}^m \times \SL_2} \& {\GL(V')}
	\arrow["{\Ad \varphi}", from=1-1, to=1-2]
	\arrow[from=1-1, to=2-1]
	\arrow[from=1-2, to=2-2]
	\arrow["{\Ad' \varphi'}"', from=2-1, to=2-2]
\end{tikzcd}
\end{equation}
To see this, let $(\sigma,x) \in W_F/I_F^m \times \SL_2$. We need to verify the following equation:
\begin{equation}\label{Eq: Adphi_commutes}
    \bar{d\psi_m} \Ad \varphi(\sigma,x) = \Ad'\varphi'(\Del_m(\sigma),x) \bar{d\psi_m}.
\end{equation}
First, write $\varphi(\sigma,x) = (\sigma_1,x_1) \in {}^L G_m$.
Then $\varphi'(\Del_m(\sigma),x) = (\Del_m(\sigma_1), \psi_m(x_1))$. The right-hand side of Equation~\ref{Eq: Adphi_commutes} becomes
\[
\Ad_{\hat{G}'}(\psi_m(x_1)) d(\Del_m(\sigma_1)) \bar{d\psi_m}.
\]
Using Equations~\ref{Eq:dsigma_commutes_with_dpsi} and~\ref{Eq:Ad_commutes_with_dpsi}, we see that this equals the left-hand side of Equation~\ref{Eq: Adphi_commutes}.

 Let $\psi'$ be an additive character of $F'$ such that $\psi \sim_m \psi'$. Fix additive measures on $F$ and $F'$ such that $\Vol(\mathfrak{o}_F) = \Vol(\mathfrak{o}_{F'})$. We have the adjoint $L$-function $L(\varphi,V,s)$, the adjoint gamma factor $\gamma(\varphi,V,s)$, and the epsilon factor $\epsilon(\varphi,V,s)$ as in Subsection~\ref{sec:Langland_parameters}. We prove the following proposition:

\begin{proposition}\label{L_functions_of_reps_are_equal}
    We have $L(\varphi,V,s) = L(\varphi',V',s)$, $\gamma(\varphi,V,s) = \gamma(\varphi',V',s)$, and $\epsilon(\varphi,V,s) = \epsilon(\varphi',V',s)$.
\end{proposition}

\begin{proof}
   Recall that we have $V = \bigoplus_{n \ge 0} V_n \otimes \Sym^n$ and a similar expression for $V'$. Note that $V_n = \Hom_{\SL_2}(\Sym^n, V)$. The action of $W_F$ on $V_n$ is given by $(\sigma \cdot f)(x) := \Ad \varphi(\sigma, 1)(f(x))$, where $\sigma \in W_F$, $f \in V_n$, and $x \in \Sym^n$. 
The isomorphism $\bar{d\psi_m} \colon V \simeq V'$ induces an isomorphism $\alpha_n \colon V_n \to V_n'$ given by 
\[
\alpha_n(f) = \bar{d\psi_m} \circ f.
\]The commutativity of Diagram~\ref{fig:Adphi_commutes} implies that the following diagram commutes:
\[
\begin{tikzcd}[ampersand replacement=\&]
    W_F/I_F^m \& \GL(V_n) \\
    W_{F'}/I_{F'}^m \& \GL(V_n')
    \arrow[from=1-1, to=1-2]
    \arrow["\Del_m"', from=1-1, to=2-1]
    \arrow[from=1-2, to=2-2]
    \arrow[from=2-1, to=2-2]
\end{tikzcd}
\]
For each $n \ge 0$, we have $V_n^{I_F} = V_n^{I_F/I_F^m}$, and $\alpha_n$ induces an isomorphism $V_n^{I_F} \simeq {V_n'}^{I_{F'}}$ such that the following diagram commutes:
\[
\begin{tikzcd}[ampersand replacement=\&]
    V_n^{I_F} \& {V_n'}^{I_{F'}} \\
    V_n^{I_F} \& {V_n'}^{I_{F'}}
    \arrow["\sim", from=1-1, to=1-2]
    \arrow["\varphi(\Fr)"', from=1-1, to=2-1]
    \arrow["\varphi'(\Fr')", from=1-2, to=2-2]
    \arrow["\sim"', from=2-1, to=2-2]
\end{tikzcd}
\]
It follows that
\[
\det\left(1 - q^{-s - \frac{n}{2}} \varphi(\Fr) \;\middle|\; V_n^{I_F}\right)^{-1} = \det\left(1 - q^{-s - \frac{n}{2}} \varphi'(\Fr') \;\middle|\; {V_n'}^{I_{F'}}\right)^{-1},
\]
and thus
\[
L(\varphi,V,s) = L(\varphi',V',s).
\]

The fact that $\alpha_n \colon V_n \to V_n'$ is $\Del_m$-equivariant implies that the representation $V_n \circ \Del_m^{-1}$ is isomorphic to $V_n'$. Because local epsilon factors are equal for isomorphic representations, it follows that the epsilon factors of $V_n'$ and $V_n \circ \Del_m^{-1}$ are equal. The equality of the local epsilon factors of $V_n$ and $V_n \circ \Del_m^{-1}$ follows from an argument similar to that of Proposition~3.7.1 of \cite{Deligne84}, applied to Weil groups instead of absolute Galois groups. It follows that $\omega(V_n) = \omega(V_n')$ and $\omega(\varphi,V) = \omega(\varphi',V')$. The proof of the equality of Artin conductors $a(V_n) = a(V_n')$ follows from Remark~\ref{Remark: Artin_cond_of_Weil_group_over_close_fields}. Hence, $\alpha(\varphi,V) = \alpha(\varphi',V')$. Consequently, we have $\epsilon(\varphi,V,s) = \epsilon(\varphi',V',s)$. Using the fact that $V$ is self-dual, along with the equality of the adjoint $L$-functions and epsilon factors, the gamma factors (defined by Equation~\ref{Eq:adjoint_gamma_factor}) satisfy $\gamma(\varphi,V,s) = \gamma(\varphi',V',s)$.
\end{proof}

\section{Formal degree conjecture over close local fields}\label{Frml_deg_conj_ovr_cls_flds}

The aim of this section is to prove that the formal degree conjecture holds if the Deligne-Kazhdan theory of close local fields is compatible with the LLC. Let $\Pi(G,F)_m$ be the set of isomorphism classes of representations $\pi$ of $G(F)$ such that $\depth(\pi) \le m$. Let $\Pi^2(G,F)_m$ be the set of discrete series representations in $\Pi(G,F)_m$. For the rest of the section, we assume there exists $l \ge m$ with $F \sim_l F'$ such that $\Kaz_m \colon \Pi(G,F)_m \to \Pi(G',F')_m$ is a bijection, and that the following diagram commutes:
\begin{equation}\label{LLC_compatibility_DKT}
\begin{tikzcd}[ampersand replacement=\&]
    {\Pi(G,F)_m} \& {\mathcal L(G,F)_m} \\
    {\Pi(G',F')_m} \& {\mathcal L(G',F')_m}
    \arrow["\LLC", from=1-1, to=1-2]
    \arrow["{\Kaz_m}"', from=1-1, to=2-1]
    \arrow["{\Del_m}", from=1-2, to=2-2]
    \arrow["\LLC"', from=2-1, to=2-2]
\end{tikzcd}
\end{equation} and the following equality holds:
\begin{equation}\label{Eq:dim_of_irred_comp}
    \dim(\rho_{\pi}) = \dim(\rho_{\pi'}).
\end{equation}

If $F \sim_m F'$, let $\psi'$ be an additive character of $F'$ such that $\psi \sim_m \psi'$. 
Let $\pi$ be a discrete series representation of $G(F)$, and let $\varphi = \LLC(\pi)$ be the corresponding parameter. Let $\pi'$ and $\varphi'$ be obtained based on commutative Diagram~\ref{LLC_compatibility_DKT} above. We have the following:

\begin{theorem}\label{thm:Formaldeg_conjecture_over_close_fields}
    Set $m \ge 1$. There exists $l \ge m$ such that if $F \sim_l F'$, then the formal degree conjecture holds for $(F, G, \varphi, \pi, \psi)$ if and only if it holds for $(F', G', \varphi', \pi', \psi')$.
\end{theorem}

\begin{proof}
Suppose the conductor of $\psi$ is $k$. Then the self-dual Haar measure $\mu_{\psi}$ on $F$ satisfies $\mu_{\psi}(\mathfrak{o}_F) = q^{k/2}$. Because the character $\psi'$ has conductor $k$, we have $\mu_{\psi'}(\mathfrak{o}_{F'}) = q^{k/2}$. Therefore, $\mu_{\psi}(\mathfrak{o}_F) = \mu_{\psi'}(\mathfrak{o}_{F'})$.

By Theorem~\ref{formal_degrees_are_equal1}, there exists $l \ge m$ such that whenever $F \sim_l F'$, the Hecke algebra isomorphism $\Kaz_m \colon \mathcal{H}(G,P_m) \cong \mathcal{H}(G',P_m')$ implies the equality $d(\pi, d\bar{g}) = d(\pi', d\bar{g}')$. There exist $c, c' > 0$ such that $\mu_{Z \backslash G} = c \, d\bar{g}$ and $\mu_{Z' \backslash G'} = c' \, d\bar{g}'$.
We wish to show that
\begin{equation}\label{eq: formal_deg_equal_wrt_gross_measure}
    d(\pi, \mu_{Z \backslash G}) = d(\pi', \mu_{Z' \backslash G'}).
\end{equation}
To show Equation~\ref{eq: formal_deg_equal_wrt_gross_measure}, it is sufficient to prove that $c = c'$.

Recall that $d\bar{g}$ is a measure on $Z(F) \backslash G(F)$ such that $dg = dz \, d\bar{g}$, where $dg$ and $dz$ satisfy $\Vol(P_m; dg) = 1 = \Vol(Z(F) \cap P_m; dz)$. Let $d\nu$ be the measure on $G(F)$ such that $d\nu = dz \, \mu_{Z \backslash G}$, and similarly define $d\nu'$ on $G'(F')$. Then
\[
\Vol(P_m; d\nu) = \Vol(Q_m; \mu_{Z \backslash G}),
\]
where $Q_m = Z(F) \backslash P_m Z(F)$. By Theorem~\ref{Thm:Parahoric_subgroups_of_quotients}, Corollary~\ref{Measures_are_equal_close_field_type1}, and the fact that $a(M_{G/Z}) = a(M_{G'/Z'})$, it follows that
\[
\Vol(Q_m; \mu_{Z \backslash G}) = \Vol(Q_m'; \mu_{(Z\backslash G)'}).
\]
Also, we have
\[
\Vol(Q_m; \mu_{Z \backslash G}) = c \Vol(Q_m; d\bar{g}) = c \Vol(P_m; dg) = c.
\]
Similarly,
\[
\Vol(Q_m'; \mu_{(Z \backslash G)'}) = c'.
\]
Hence, $c = c'$.

 Proposition \ref{param_over_close_fields} and \ref{L_functions_of_reps_are_equal} imply $\# \mathcal S_{\varphi}^{\sharp} = \# \mathcal{S_{\varphi'}^{\sharp}}$ and $\gamma(\varphi) = \gamma(\varphi').$ By our assumption \ref{Eq:dim_of_irred_comp}, it follows that $\dim \rho_{\pi} = \dim \rho_{\pi'}.$ Thus, the theorem follows.

\end{proof}

\begin{remark}
   The validity of the equation $d(\pi, \mu_{Z\backslash G}) =\frac{\dim(\rho_{\pi})}{\#\mathcal{ S }_{\varphi}^\sharp} \vert \gamma(\varphi) \vert$ does not depend on the conductor $k$~\cite[Lemma 1.1, Lemma 1.3]{HiragaIchinoKaoru08}.
\end{remark}

\section{Formal degree conjecture over local function fields}

In this section, let $G$ be one of the following split groups over $F$ of characteristic $p > 0$:
\begin{enumerate}
    \item $\GL_n$;
    \item $\Sp_{2n}$, $\SO_{2n}$, and $\mathrm{SO}_{2n+1}$, with $p>2$;
    \item $\GSp_4$, with $p > 2$.
\end{enumerate}

The commutativity of Diagram~\ref{LLC_compatibility_DKT} can be established as follows. In the following papers, the commutativity of the diagram below is established:
\begin{itemize}
    \item for $\GL_n$, in~\cite[Section 7]{Ganapathy15};
    \item for $\GSp_4$ ($p > 2$), in~\cite[Section 9.3]{Ganapathy15};
    \item for $\Sp_{2n}$, $\SO_{2n}$, and $\mathrm{SO}_{2n+1}$ ($p>2$), in~\cite[Section 13.6]{GanapathyVarma17}.
\end{itemize} 
\begin{equation}
\begin{tikzcd}[ampersand replacement=\&]
    {\Pi(G,F)_m} \& {\mathcal L(G,F)_m} \\
    {\Pi(G',F')_m} \& {\mathcal L(G',F')_m}
    \arrow["\LLC", from=1-1, to=1-2]
    \arrow["{\zeta_m}"', from=1-1, to=2-1]
    \arrow["{\Del_m}", from=1-2, to=2-2]
    \arrow["\LLC"', from=2-1, to=2-2]
\end{tikzcd}
\end{equation}
Here, $\zeta_m : \mathcal H(G(F), I_m) \cong \mathcal{H} (G(F'), I_m')$ is the Iwahori variant of the Kazhdan isomorphism in~\cite[Theorem 3.13]{Ganapathy15}, and $F'$ is a field of characteristic $0$ that is sufficiently close to $F$. The map $\zeta_m$ satisfies $\zeta_m \vert_{\mathcal H(G(F),P_m)} = \Kaz_m$. By Theorem~\ref{KazIsomorphism}, there exists $l \ge m$ such that if $F \sim_l F'$, then $\Kaz_m : \mathcal H(G,P_m) \to \mathcal H(G', P_m')$ is an isomorphism. Combining all of these, one concludes the commutativity of Diagram~\ref{LLC_compatibility_DKT}.

Note that for any group $G$ as above and for any discrete parameter $\varphi$, the group $\pi_0(\overline{S_{\varphi}})$ is a finite abelian group. In particular, for $\GL_n$, it is trivial. For $\GSp_4$, it is either trivial or isomorphic to $\mathbb{Z}/2\mathbb{Z}$~\cite[Lemma~6.2]{GanTakeda11}. For $\Sp_{2n}$ and $\SO_n$, it is an elementary abelian $2$-group~\cite[Section~4]{GanGrossPrasad12}. As noted in Remark~\ref{Rem: component_grp_abelian}, it is sufficient to prove that
$$d(\pi, \mu_{Z\backslash G}) =\frac{1}{\#\mathcal{ S }_{\varphi}^\sharp} \vert \gamma(\varphi) \vert.$$
This has been established over fields of characteristic $0$ for the following groups:
\begin{itemize}
    \item $\GL_n$, in~\cite[Theorem 3.1]{HiragaIchinoKaoru08};
    \item $\SO_{2n+1}$, in~\cite[Corollary 5.1]{IchinoAtsushietc17};
    \item $\GSp_4$, in~\cite[Theorem 1.3]{GanIchino14}; and
    \item $\Sp_{2n}$ and $\SO_{2n}$, in~\cite[Theorem 1.1]{beuzartplessis2025}.
\end{itemize}

Combining Theorem~\ref{thm:Formaldeg_conjecture_over_close_fields} with this result, we have the following theorem: 

\begin{theorem}\label{thm:Formal_deg_over_function_field}
    Let $\varphi$ be a discrete parameter. Let $\pi \in \Pi^2_{\varphi}(G,F)$. Then 
    $$d(\pi, \mu_{Z \backslash G}) =  \frac{1}{\#\mathcal{ S }_{\varphi}^\sharp} \vert \gamma(\varphi) \vert.$$ 
\end{theorem}

\section*{Acknowledgments}

The author would like to thank his supervisor, Radhika Ganapathy, for suggesting this problem and for her constant guidance. He is also grateful to her for carefully reading an earlier draft and offering many helpful corrections. The author is supported by the NBHM doctoral fellowship, Ref. No. 0203/12/2023/R\&D-II/16573.

\end{document}